\documentclass[oneside]{amsart}
\usepackage[english]{babel}
\usepackage[T1]{fontenc}
\usepackage[latin9]{inputenc}
\usepackage{units}
\usepackage{mathrsfs}
\usepackage{mathtools}
\usepackage{amsbsy}
\usepackage{amsthm}
\usepackage{amssymb}
\usepackage{setspace}
\usepackage{esint}
\usepackage{xcolor}
\usepackage{makecell}
\usepackage{float}
\usepackage{upgreek}
\usepackage{enumitem}
\usepackage{pgfplots}

\usepackage{bbm} 
\usepackage[unicode,psdextra,hidelinks]{hyperref}
\makeatletter
\numberwithin{equation}{section}
\numberwithin{figure}{section}

\newtheorem{theorem}{Theorem}[section]
\newtheorem*{theorem*}{Theorem}
\newtheorem{corollary}[theorem]{Corollary}
\newtheorem{lemma}[theorem]{Lemma}
\newtheorem{proposition}[theorem]{Proposition}

\theoremstyle{definition}
\newtheorem{definition}{Definition}

\theoremstyle{remark}
\newtheorem{remark}[theorem]{Remark}

\makeatother

\usepackage{babel}

\newcommand{\SL}{\operatorname{SL}}

\newcommand{\diag}{\operatorname{diag}} 

\newcommand{\Lie}{\operatorname{Lie}}

\newcommand{\dtt}{\mathtt{d}} 

\newcommand{\att}{\mathtt{a}}

\newcommand{\SLdRZ}{\SL_{\dtt}(\mathbb{R})/\SL_{\dtt}(\mathbb{Z})}

\newcommand{\RBS}{\operatorname{RBS}} 

\newcommand{\stab}{\operatorname{stab}} 

\newcommand{\Hull}{\operatorname{Hull}} 

\newcommand{\all}{\operatorname{all}}

\newcommand\mathitem{\item\leavevmode\vspace*{-\dimexpr\baselineskip+\abovedisplayskip\relax}}

\begin{document}

\title[Upper bounds for the entropy in the cusp]{Upper bounds for the entropy in the cusp for one-parameter diagonal flows on $\SL_{\mathtt{d}}(\mathbb{R})/ \SL_{\mathtt{d}}(\mathbb{Z})$}
\author{Ron Mor}

\begin{abstract}
    We give explicit upper bounds for the entropy in the cusp for one-parameter diagonal flows on $\SL_{\dtt}(\mathbb{R})/\SL_{\dtt}(\mathbb{Z})$. These results include bounds for the entropy of the cusp as a whole, as well as for the cusp regions corresponding to either the maximal parabolic subgroups of $\SL_{\dtt}$, or the minimal (Borel) parabolic subgroup.
    To do so, we use a method which involves choosing an auxiliary linear functional on the Lie algebra of the Cartan group, specifically tailored to each of the bounds we are interested in.
    In a follow-up paper we prove
    that the upper bounds we obtain in this work are tight.
\end{abstract}
\thanks{This work was supported by ERC 2020 grant HomDyn (grant no.~833423)}
\address{The Einstein Institute of Mathematics\\
	Edmond J. Safra Campus, Givat Ram, The Hebrew University of Jerusalem
	Jerusalem, 91904, Israel}
	
\maketitle

\section{Introduction}
\subsection{Motivation and background}\label{subsec:1.1}
This paper is the second in a sequence of papers in which we study the entropy of
any one-parameter diagonal subgroup $\att_{\bullet}=\{\mathbf{a}_{t}\}_{t\in\mathbb{R}}$ of $\SL_{\mathtt{d}}(\mathbb{R})$, with positive entries on the diagonal, on the space $\SL_{\mathtt{d}}(\mathbb{R})/ \SL_{\mathtt{d}}(\mathbb{Z})$ of unimodular lattices in $\mathbb{R}^{\mathtt{d}}$.
In this paper, we use the results and methods of~\cite{mor2022a} to deduce explicit
upper bounds for the entropy of the cusp, corresponding to certain cusp regions separately and to the cusp as a whole. 
In a follow-up paper~\cite{mor2022c} we prove these upper bounds are tight.

The question of estimating the entropy of the cusp, i.e.\ the highest amount of entropy which can be carried by a sequence of invariant probability measures which converges to the zero measure, was studied in various settings by many authors, especially considering upper bounds.
For a more detailed introduction to these works, see~\cite{mor2022a}.
Some of these works computed upper bounds for the entropy of measures before going to the limit, and then applied these results to bound the limit entropy.
These include the paper~\cite{einsiedler2011distribution} for the study of $\SLdRZ$ in the rank $1$ case $\dtt=2$, as well as the papers~\cite{einsiedler2012entropy,kadyrov2017singular} for the higher rank case $\dtt\geq 3$ when the action $\att_{t}$ has only two distinct eigenvalues -- first in \cite{einsiedler2012entropy} for the action of $\att_{t}=\diag(e^{t/2},e^{t/2},e^{-t})$ on $\SL_{3}(\mathbb{R})/\SL_{3}(\mathbb{Z})$, and later in~\cite{kadyrov2017singular} for the action of \[\att_{t}=\diag(e^{nt},\ldots,e^{nt},e^{-mt},\ldots,e^{-mt})\] on $\SL_{m+n}(\mathbb{R})/\SL_{n+m}(\mathbb{Z})$ for general $m,n\geq 1$.
This approach was also used in~\cite{einsiedler2015escape} to study diagonal actions on spaces $G/\Gamma$ where $G$ is a connected semisimple real Lie group of rank 1 with
finite center and $\Gamma$ is a lattice,
and in~\cite{mor2021excursions} to study the frame flow on geometrically finite hyperbolic orbifolds. 
This is also the direction we chose in~\cite{mor2022a} to study the action of any one-parameter diagonal subgroup with positive entries on the diagonal on $\SLdRZ$. 
Other works considered the entropy in the cusp directly, such as~\cite{iommi2015entropy,riquelme2017escape} for the entropy in the cusp for geometrically finite Riemannian manifolds with pinched negative sectional curvature
and uniformly bounded derivatives of the sectional curvature.

Another type of questions discussed in previous works is the study of the Hausdorff dimension of the set of singular systems of linear forms. It is related to dynamics of one-parameter diagonal groups acting on $\SLdRZ$ by a relation known as Dani's correspondence~\cite{dani1985divergent}, and its generalizations by Kleinbock~\cite{kleinbock1998bounded,kleinbock1998flows}. 
In particular we note the works~\cite{cheung2011hausdorff,cheung2016hausdorff,das2019variational} for the unweighted case (where the acting transformation $\att_{t}$ has only two distinct eigenvalues), and the work~\cite{liao2016hausdorff} for general one-parameter actions in the particular case $\dtt=3$.
This aspect was also discussed in the papers~\cite{einsiedler2011distribution,kadyrov2017singular} via Dani's correspondence.
A notable recent work in this direction by Solan 
is~\cite{solan2021parametric},
where
the Hausdorff dimension with respect to the expansion metric was studied, for general one-parameter diagonal groups in higher rank. These results 
are related to one of the results we report in this work, namely the 
computation of the entropy for the cusp as a whole, as will be mentioned later on. 
We make use of some of the results from~\cite{solan2021parametric} in our follow-up paper~\cite{mor2022c} to prove
our estimates are sharp.

\medskip
In this paper, we find upper bounds for the entropy of the cusp, for general one-parameter diagonal actions on $\SLdRZ$. 
We make use of a new method developed in~\cite{mor2022a},
which allows us to include an auxiliary linear functional in the entropy computations. Apart from giving new results for the entropy before going to the limit~\cite{mor2022a}, 
this method allows for the entropy in the cusp to be computed in a simple combinatorical manner, 
without having to account for all of the different possible trajectories in the cusp.

In order to apply this method, one needs to choose carefully the linear functional to include, in order to optimize the bound we get for the entropy. This is the goal of this paper. In a follow up paper~\cite{mor2022c} we will show that our estimates are optimal (c.f.\ \cite[Theorem 1.6]{mor2022a}).
As in~\cite{mor2022a}, we especially consider
the finer structure of the cusp in
these higher rank spaces, stemming from
the natural compactification of $G/\Gamma$ --- the reductive Borel-Serre compactification --- where one adds a subvariety at infinity to $G/\Gamma$ for each standard parabolic proper
subgroup of $G$. 
In this paper (as well as \cite{mor2022a,mor2022c}) we estimate the cusp entropy originating from the different cusp regions separately, rather than only
for the cusp as a whole, a feature that was not discussed in the previous literature
on the subject.
We find explicit upper bounds for the entropy in three main cases: (\emph{i}) for the cusp as a whole, (\emph{ii}) for the regions corresponding to the maximal parabolic subgroups, and (\emph{iii}) for the region of the minimal (Borel) parabolic subgroup.

\subsection{Basic setup}\label{section:introduction}

Let $G=\SL_{\mathtt{d}}(\mathbb{R})$ and $\Gamma=\SL_{\mathtt{d}}(\mathbb{Z})$. 
Let $A\leq G$ be the subgroup of diagonal matrices with positive entries on the diagonal, and let $\{\att_{t}\}_{t\in\mathbb{R}}\leq A$ be a one-parameter diagonal flow on $G/\Gamma$, namely $\att_t=\exp(t\mathbf{\upalpha})$ for all $t$, for some $\upalpha\in \Lie(A)$. We write $\upalpha=\diag(\upalpha_1,\ldots,\upalpha_\dtt)$ with $\sum_{i=1}^{\dtt} \upalpha_i=0$.

Let $\mathcal{P}$ be the set of 
standard $\mathbb{Q}$-parabolic subgroups of $\SL_{\mathtt{d}}$. Recall that these subgroups 
are the groups of upper-triangular block matrices with real entries and determinant $1$.

Let $T$ be the full diagonal subgroup of $G$. 
For any $P\in\mathcal{P}$ let $W(T,P)=N_P(T)/C_{P}(T)$, where $N_P(T)$ and $C_{P}(T)$ are the normalizer and centralizer of $T$ in $P$, respectively.
In particular, $W=W(T,G)$ is the Weyl group of $G$. 
We consider the right action of $W$ on $T$ by conjugation, where the action of $w=nC_{G}(T)$ on $a\in T$ is denoted by $a^{w}$ and is given by $a^{w}\coloneqq n^{-1}an$, independently of the representative $n$. 
In our settings, 
$W$ is isomorphic to the symmetric group $S_{\mathtt{d}}$,
and
the action of $\sigma\in S_{\dtt}\cong W$ on $A$ reads as 
\[\diag(a_1,\ldots,a_{\dtt})^{\sigma}=\diag(a_{\sigma_1},\ldots,a_{\sigma_\dtt}).\]
Note that $W$ acts on $\Lie(A)$ by permutations as well.

We consider quotients of $W$ by two different subgroups. First, in the case that there are multiplicities in the multiset of eigenvalues of $\att_{t}$, different elements in the Weyl group can conjugate $\att_{t}$ to the same element. To compensate for this it will be convenient to divide $W$ by 
$\stab_{W}(\att)$, i.e.\ the 
stabilizer in $W$ of the
time-one map $\att\coloneqq \att_{1}=\exp(\upalpha)$.
Next, for a parabolic subgroup $P\in\mathcal{P}$, we also divide $W$ by the subgroup $W(T,P)$ of permutations which preserve the block structure of $P$, in the sense that $\sigma,\tau\in W$ are identified in $W/W(T,P)$
if their actions on any $a\in A$ are the same up to a permutation which preserves the multiset of values of $a$ in each block of $P$. 
Finally, we define the double-quotient 
\[W_{P,\att}=\stab_{W}(\att)\backslash W/W(T,P).\]
For $w\in W$, we let $[w]_P$ stand for the double-coset in $W_{P,\att}$ corresponding to $w$ (since $\att$ is considered fixed, we omit it from this notation).

Consider any subgroup $H_{S}\subseteq G$ of the form \[H_{S}=(I+\bigoplus_{(i,j)\in S} U_{ij}) \cap G\] for some set $S\subseteq\{1,\ldots,\dtt\}^{2}$, where $U_{ij}$ is the set of matrices with zeros at all entries except possibly the $(i,j)$th entry. This includes the parabolic subgroups of $G$. 
For such a subgroup $H_{S}$, and for any diagonal matrix $a=\exp(\alpha)\in A$, where $\alpha=\diag(\alpha_1,\ldots,\alpha_{\mathtt{d}})\in \Lie(A)$, we define the entropy of $a$ on $H_{S}$ as the sum of (positive) Lyapunov exponents
\begin{align*}
    &h(H_S,a)=\sum_{(i,j)\in S} (\alpha_i-\alpha_j)^{+}, 
\end{align*}
where $z^{+}=\max(z,0)$ for any $z\in\mathbb{R}$.

Next, for a parabolic subgroup $P\in\mathcal{P}$, let $A_P<A$ be the identity component of the centralizer of the Levi part of $P$, or more concretely the subgroup of $P$ consisting of block scalar matrices on the diagonal, with positive entries. 
For any $P\in\mathcal{P}$ and $[w]_{P}\in W_{P,\att}$ we define 
the projection of $\upalpha^{w}$ from $\Lie(A)$ to $\Lie(A_P)$ by
\[\pi_P(\upalpha^{w})=\frac{1}{| W(T,P)|}\sum_{u\in W(T,P)}\upalpha^{wu}.\]
Then,
for $\phi\in \Lie(A)^{\ast}$, we define
\begin{equation*}(h-\phi)([w]_P)=h(P,\att^{w})-\phi(\pi_P(\upalpha^{w})).\end{equation*}
Note that $h(P,\att^{w})$ and $\pi_{P}(\upalpha^{w})$ depend only on the coset $[w]_{P}$ rather than on $w$.

\subsection{Entropy in the cusp}\label{subsec:1.33}
The entropy in the cusp is defined as
\[h_{\infty}(\att)=\sup\left\{\limsup_{i\to\infty} h_{\mu_i}(\att):\ \mu_i\in M_{1}(G/\Gamma)\ \att\text{-invariant},\ \mu_i\rightharpoonup 0\right\},\]
where for any locally compact metric space $Y$ we use $M_{1}(Y)$ to denote the space of Borel probability measures on $Y$.

In~\cite{mor2022a} we introduced the following more refined notions for the entropy in the cusp, using the reductive Borel-Serre (RBS) compactification of $G/\Gamma$~\cite{borel2006compactifications}. This compactification may be described as a disjoint union \[\overline{G/\Gamma}^{\RBS}=\coprod_{P\in\mathcal{P}}e_{\infty}(P).\] The set $e_{\infty}(G)$ in the above disjoint union is equal to $G/\Gamma$, and is open and dense according to the topology of $\overline{G/\Gamma}^{\RBS}$. Without getting into details (see~\cite{borel2006compactifications,mor2022a} for more), in this compactification the part of $G/\Gamma$ which is close to the boundary component $e_{\infty}(P)$ is the set of lattices which have a unique flag of linear subspaces of $\mathbb{R}^{\dtt}$ of small covolume, with $P$ determining the dimensions of the subspaces in the flag.
Then, we define the entropy in the cusp region associated to $P\in\mathcal{P}$ by
\[h_{\infty,P}(\att)=\sup\left\{\limsup_{i\to\infty} h_{\mu_i}(\att):\ \mu_i\in M_{1}(G/\Gamma)\ \att\text{-invariant},\ \mu_i\rightharpoonup \nu\in M_{1}(e_{\infty}(P))\right\}.\]
Furthermore, we can consider a similar notion for the entropy in the cusp associated to all parabolic subgroups contained in $P$, given by
\begin{align*}
h_{\infty,\subseteq P}(\att)&=\sup\bigg\{\limsup_{i\to\infty} h_{\mu_i}(\att):\ \mu_i\in M_{1}(G/\Gamma)\ \att\text{-invariant},\\& 
\hspace{7.1cm}\mu_i\rightharpoonup \nu\in M_{1}(\bigcup_{Q\subseteq P}e_{\infty}(Q))\bigg\}.\end{align*}

In~\cite{mor2022a} we proved the following result regarding bounds for $h_{\infty}(\att)$, $h_{\infty,P}(\att)$ and $h_{\infty,\subseteq P}(\att)$, using auxiliary linear functionals. This is the key ingredient we use in this current paper.
\begin{theorem}[\cite{mor2022a}]\label{theorem:1.0}
Let $\att=\exp(\upalpha)\in A$, and $\phi\in \Lie(A)^{\ast}$. Then
\begin{enumerate}
    \item The entropy in the cusp is bounded by
\[h_{\infty}(\att)\leq \max_{P\in\mathcal{P}\smallsetminus\{G\}}\max_{[w]_P\in W_{P,\att}}\Big(h(P,\att^w)-\phi(\pi_{P}(\upalpha^{w}))\Big).\]
\item
For all parabolic subgroups $P\in \mathcal{P}$,
\[\mathit{(i)}\qquad h_{\infty,P}(\att)\leq \max_{[w]_P\in W_{P,\att}}\Big(h(P,\att^w)-\phi(\pi_{P}(\upalpha^{w}))\Big)\]
\[\mathit{(ii)}\qquad h_{\infty,\subseteq P}(\att)\leq \max_{\substack{Q\in\mathcal{P}:\\ Q\subseteq P}}\max_{[w]_Q\in W_{Q,\att}}\Big(h(Q,\att^w)-\phi(\pi_{Q}(\upalpha^{w}))\Big)\]    
\end{enumerate}
\end{theorem}

\subsection{Results}\label{sec:introduction}

As mentioned above, in this paper we find specific linear functionals, which depend on the flow $\att_{\bullet}$, and are suitable to maximize the entropy in the cusp for three cases: (\emph{i}) $h_{\infty}(\att)$, (\emph{ii}) $h_{\infty,\subseteq P}(\att)$ for maximal parabolic subgroups $P$ (and consequently $h_{\infty,P}(\att)$ as well), and (\emph{iii}) $h_{\infty,B}(\att)$ for the (minimal) Borel subgroup $B\in\mathcal{P}$. In a follow-up paper~\cite{mor2022c} we prove that the upper bounds on entropy are sharp, as well as that $h_{\infty,P}(\att)=h_{\infty,\subseteq P}(\att)$ for maximal parabolic groups $P$.

Let us state our results.
The most interesting structure occurs for the case of maximal parabolic subgroups as discussed in the following theorem.
Here $P_k$ stands for the (maximal parabolic) group
of upper triangular block matrices with two blocks, the first one being of size $k$ and the other of size $\dtt-k$.
\begin{theorem}\label{theorem:1.1}
Let $\att=\exp(\upalpha)\in A$.
Assume without loss of generality that $\upalpha_{i}\geq \upalpha_{i+1}$ for all $1\leq i\leq \dtt-1$. 
For all $1\leq k\leq \dtt/2$, let $m_k$ be the minimal integer such that \[\sum_{i=1}^{k} \upalpha_{m_k+2(i-1)}\geq0\geq\sum_{i=1}^{k}\upalpha_{m_k+2(i-1)+1}.\]
Then
    \begin{equation}\label{eq:1.1}h_{\infty,\subseteq P_{k}}(\att)\leq h(G,\att)-k\sum_{i=1}^{m_k}\upalpha_{i}-\sum_{i=1}^{k-1}(k-i)(\upalpha_{m_k+2i-1}+\upalpha_{m_k+2i}).\end{equation}
Furthermore, the same upper bound holds for $h_{\infty,\subseteq P_{\dtt-k}}(\att)$ as well.
\end{theorem}
Note that after some formal manipulations, the right hand side of Equation~\eqref{eq:1.1} equals 
\[
h(G,\att)-
\sum_{i=1}^{k-1}(k-i)(\upalpha_{2i-1}+\upalpha_{2i})-\sum_{j=1}^{m_k}\sum_{i=1}^{k}\upalpha_{j+2i-2}
\]
which was the form initially announced in \cite{mor2022a}.

As can be deduced from Equation~\eqref{eq:1.1}, the upper bound we found for $h_{\infty,\subseteq P_k}(\att)$ is monotone decreasing with $1\leq k\leq \dtt/2$, so that the largest bounds are obtained for the maximal parabolic subgroups $P_1$ and $P_{\dtt-1}$. Considering that the special constant $m_1$ is equal to the number of positive entries of $\upalpha$, we deduce that 
\[h_{\infty,\subseteq P_1}(\att)\leq h(G,\att)-\sum_{i=1}^{\dtt}\upalpha_{i}^{+},\]
and similarly for $h_{\infty,\subseteq P_{\dtt-1}}(\att)$.
Our next theorem asserts that this value is an overall bound for the entropy in the cusp, without limitation on the parabolic subgroups. 
This is the result we mentioned has some similarities with Solan's work~\cite{solan2021parametric}.
\begin{theorem}\label{theorem:1.3}
Let $\att=\exp(\upalpha)\in A$. Then
\[h_{\infty}(\att)\leq h(G,\att)-\sum_{i=1}^{\dtt}\upalpha_{i}^{+}.\]
\end{theorem}

Lastly, we found a linear functional which is suitable for the minimal (Borel) parabolic subgroup rather than the maximal parabolic subgroups.
This functional averages out the entropy for all the different elements in the Weyl group, giving a uniform entropy estimate of $\frac{1}{2}h(G,\att)$ to all directions. This generalizes the factor half obtained in~\cite{einsiedler2011distribution} for the entropy of the cusp in the $\dtt=2$ case. 
\begin{theorem}\label{theorem:1.4}
Let $B\in\mathcal{P}$ be the Borel subgroup, and let $\att\in A$. Then
\[h_{\infty,B}(\att)\leq\frac{1}{2}h(G,\att).\]
\end{theorem}

\subsection{Structure of the paper}
The paper is organized as follows. In~\S\ref{sec:borel} we present a short proof of the Borel result (Theorem~\ref{theorem:1.4}). Next, the result for maximal parabolic groups (Theorem~\ref{theorem:1.1}) is shown in~\S\ref{section:h_for_pmax}. Lastly, Theorem~\ref{theorem:1.3}, regarding the entropy of the entire cusp, is proved in~\S\ref{section:4} using some of the ideas of~\S\ref{section:h_for_pmax}.

\subsection{Acknowledgements} This paper is a part of the author's PhD studies in The Hebrew University of Jerusalem. I would like to thank my advisor, Prof.\ Elon Lindenstrauss, to whom I am grateful for his guidance and support throughout this work.

\section{Bounding entropy for the Borel subgroup}\label{sec:borel}
We begin with proving Theorem~\ref{theorem:1.4}, namely finding an upper bound for $h_{\infty,B}(\att)$.
To do so, we use a linear functional $\phi_{B}$
which averages out the entropy between all the different orientations, so that $(h-\phi_{B})([w]_{B})$ is 
the same for all $w\in W$.
Here, and throughout the paper, we let $\lambda_i\in \Lie(A)^{\ast}$ be the value of the $i$th diagonal entry $\lambda_{i}(\diag(\alpha_1,\ldots,\alpha_{\dtt}))=\alpha_i$, for any $1\leq i\leq \dtt$.
\begin{proof}[Proof of Theorem~\ref{theorem:1.4}]
Note that for any $w\in W$,
\[h(B,\att^{w})=\sum_{i<j}(\upalpha_{w_i}-\upalpha_{w_j})^{+}.\]
Let 
\[\phi_{B}=\frac{1}{2}\sum_{i<j}(\lambda_{i}-\lambda_{j}).\]
Then
\begin{align*}(h-\phi_{B})([w]_{B})&=\sum_{i<j}\Big[(\upalpha_{w_i}-\upalpha_{w_j})^{+}-\frac{1}{2}(\upalpha_{w_i}-\upalpha_{w_j})\Big]\\&
=\frac{1}{2}\sum_{i<j}\Big[(\upalpha_{w_i}-\upalpha_{w_j})^{+}+(\upalpha_{w_j}-\upalpha_{w_i})^{+}\Big]
=\frac{1}{2}\sum_{1\leq i,j\leq \dtt}(\upalpha_{w_i}-\upalpha_{w_j})^{+}\\&
=\frac{1}{2}\sum_{1\leq i,j\leq\dtt}(\upalpha_i-\upalpha_j)^{+}=\frac{1}{2}h(G,\att),
\end{align*}
where we used the fact that 
\[z^{+}-\frac{1}{2}z=
\frac{1}{2}(z^{+}+(-z)^{+}),\]
for any $z\in\mathbb{R}$.
This concludes the proof, using Theorem~\ref{theorem:1.0}.
\end{proof}

\section{Bounding entropy for maximal parabolic subgroups}\label{section:h_for_pmax}
In this section we find a suitable linear functional to use for bounding the entropy of maximal parabolic subgroups $P_k$, and compute the entropy bound deduced from it as in Theorem~\ref{theorem:1.1}. 

\subsection{Conventions}\label{subsec:2.1}
We introduce conventions that will be assumed for the rest of the paper.
First, we assume without loss of generality (by possibly changing the order of the standard basis elements) that $\upalpha_i\geq \upalpha_{i+1}$ for all $1\leq i\leq \mathtt{d}-1$, i.e.\ the elements on the diagonal of $\att$ are ordered from largest to smallest. 

As discussed in~\S\ref{section:introduction}, we think of the Weyl group $W$ as the group of permutations $S_{\dtt}$.
We take the convention that for any $[w]_{P}\in W_{P,\att}$ the representative $w\in W\cong S_{\dtt}$ is chosen so that its values are increasing in each block (for example, if $P=P_{k}$ is a maximal parabolic subgroup, this means $w_{1}<\cdots<w_{k}$ and $w_{k+1}<\cdots<w_{\dtt}$). 
This ensures that the diagonal of $\att^{w}$ is monotonically decreasing in each block of $P$.
If there are multiplicities in the elements on the diagonal of $\att$ (namely $\stab_{W}(\att)\not=\{id\}$), there may be several suitable representatives, and we choose one of them arbitrarily (in this case, $\att^{w}$ does not depend on the specific choice).

For any $1\leq i\leq \dtt$, let $\lambda_i\in \Lie(A)^{\ast}$ be the value of the $i$th entry, i.e.\ $\lambda_{i}(\alpha)=\alpha_i$. For any $1\leq k\leq \dtt-1$, 
let $\psi_k=\lambda_k-\lambda_{k+1}$; these form a basis for the roots of $\Lie(A)$.
For $\sigma\in W$, let $h_k(\sigma)=h(P_{k},\att^{\sigma})$ and by abuse of notation also
\[  \psi_k(\sigma)\coloneqq\psi_{k}(\pi_{P_k}(\upalpha^{\sigma}))=\frac{1}{k}\sum_{i=1}^{k}\upalpha_{\sigma_i}-\frac{1}{\dtt-k}\sum_{i=k+1}^{\dtt}\upalpha_{\sigma_i}.
\]

\subsection{Entropy formulas}
Before diving into the bounds for the entropy in the cusp, we show in this section useful formulas for the entropy of diagonal actions on maximal parabolic subgroups.

First, note the following simple lemma, whose proof is left to the reader, which gives 
the sum of Lyapunov exponents 
for a given diagonal action. 
\begin{lemma}\label{lemma:8.4}
Let $a=\exp(\alpha)\in \SL_{N}(\mathbb{R})$ be a diagonal
matrix, so that the entries $\alpha_{i}$ on the diagonal of $\alpha$ satisfy $\alpha_{i}\geq \alpha_{i+1}$ for all $1\leq i\leq N-1$.
Then 
\[h(\SL_{N}(\mathbb{R}),a)=\sum_{i=1}^{N}(N+1-2i)\alpha_{i}.\]
\end{lemma}

For a maximal parabolic subgroup $P_{k}$, we consider its Levi subgroup $L_{P_k}$ which is the diagonal part of $P_{k}$, and its unipotent radical $U_{P_{k}}$ which is the off-diagonal part of $P_{k}$.
For convenience, we write here explicit formulas for the entropy on these subgroups.
\begin{lemma}\label{lemma:levi_computation}
    Let $a=\exp(\alpha)\in \SL_{\dtt}(\mathbb{R})$ be a diagonal matrix, and let $1\leq k\leq \dtt-1$. Then the following formulas hold:
    \begin{enumerate}
        \mathitem \[h(U_{P_k},a)=\sum_{i=1}^{k}\sum_{j=k+1}^{\dtt}(\alpha_{i}-\alpha_{j})^{+}\]
        \item If the diagonal entries of $a$ are monotonically decreasing in each block of $P_{k}$, then: \[h(L_{P_{k}},a)=\sum_{i=1}^{k}(k+1-2i)\alpha_{i}+\sum_{i=k+1}^{\dtt}(\dtt-k+1-2(i-k))\alpha_i\] 
    \end{enumerate}
\end{lemma}
\begin{proof}
    The formula for the unipotent radical is immediate by the definition of entropy (see~\S\ref{section:introduction}).

    For the Levi subgroup, let $c_1=\frac{1}{k}\sum_{i=1}^{k}\alpha_{i}$ and $c_2=\frac{1}{\dtt-k}\sum_{i=k+1}^{\dtt} \alpha_i$ be the average of the diagonal entries on the first and second blocks respectively, and let \[\overline{\alpha}_i=\begin{cases}
    \alpha_{i}-c_1 & 1\leq i\leq k \\
    \alpha_{i}-c_2 & k<i\leq \dtt
    \end{cases}.\]
    Let $\beta_1=\diag(\overline{\alpha}_1,\ldots,\overline{\alpha}_{k})$ and $\beta_2=\diag(\overline{\alpha}_{k+1},\ldots,\overline{\alpha}_{\dtt})$.
    Note that \[\sum_{i=1}^{k}\overline{\alpha}_{i}=\sum_{i=k+1}^{\dtt}\overline{\alpha}_{i}=0,\] so $\exp(\beta_{1})\in \SL_{k}(\mathbb{R})$ and $\exp(\beta_{2})\in \SL_{\dtt-k}(\mathbb{R})$.
    Then, since the Levi subgroup is composed of blocks on the diagonal, the entropy is computed as
    \begin{align*}h(L_{P_k},a)&=\sum_{i=1}^{k}\sum_{j=1}^{k}(\alpha_i-\alpha_j)^{+}+\sum_{i=k+1}^{\dtt}\sum_{j=k+1}^{\dtt}(\alpha_i-\alpha_j)^{+}\\&
    =h(\SL_{k}(\mathbb{R}),\exp(\beta_1))+h(\SL_{\dtt-k}(\mathbb{R}),\exp(\beta_2)).\end{align*}
    By Lemma~\ref{lemma:8.4},
    \[h(\SL_{k}(\mathbb{R}),\exp(\beta_1))=\sum_{i=1}^{k}(k+1-2i)(\alpha_i-c_1)=\sum_{i=1}^{k}(k+1-2i)\alpha_i,\]
    where we used the identity $\sum_{i=1}^{k}(k+1-2i)=0$.
    Similar expression holds for the entropy of $\exp(\beta_2)$, concluding the proof for the Levi subgroup.
    \end{proof}

    The following equivalent formula for the entropy in the unipotent radical will also be useful.
    \begin{lemma}\label{lemma:3.3N}
    Let $\sigma\in W$.
    For any $1\leq i\leq k$, let 
    \[P_{i}=\{j>k:\ \sigma_{j}>\sigma_{i}\},\]
    and for any $k<j\leq \dtt$ let
    \[N_{j}=\{i\leq k:\ \sigma_{i}<\sigma_{j}\}.\]
    Then
    \[h(U_{P_k},\att^{\sigma})=\sum_{i=1}^{k}|P_{i}|\upalpha_{\sigma_{i}}-\sum_{j=k+1}^{\dtt}|N_{j}|\upalpha_{\sigma_{j}}\]
    \end{lemma}
    \begin{proof}
        Let $1\leq i\leq k$ and $k<j\leq \dtt$.
        Note that the Lyapunov exponent $(\upalpha_{\sigma_{i}}-\upalpha_{\sigma_j})^{+}$ satisfies
        \[(\upalpha_{\sigma_{i}}-\upalpha_{\sigma_j})^{+}=\begin{cases}
        0& j\not\in P_{i} \\ 
        \upalpha_{\sigma_{i}}-\upalpha_{\sigma_{j}}& j\in P_{i}
        \end{cases}\]
        due to the monotonicity of $\att$.

        Furthermore, note that $j\in P_{i}\iff i\in N_{j}$, namely $\mathbbm{1}_{P_i}(j)=\mathbbm{1}_{N_{j}}(i)$, where $\mathbbm{1}_{C}(x)=\begin{cases} 1 & x\in C \\ 0 & x\not\in C\end{cases}$ is the characteristic function of a set $C$.
        Then
                \begin{align*}
            h(U_{P_k},\att^{\sigma})&=\sum_{i=1}^{k}\sum_{j=k+1}^{\dtt}(\upalpha_{\sigma_{i}}-\upalpha_{\sigma_j})^{+}=\sum_{i=1}^{k}\sum_{j=k+1}^{\dtt}\mathbbm{1}_{P_{i}}(j)(\upalpha_{\sigma_i}-\upalpha_{\sigma_j})\\&
            =\sum_{i=1}^{k}\sum_{j=k+1}^{\dtt}\mathbbm{1}_{P_{i}}(j)\upalpha_{\sigma_i}-\sum_{j=k+1}^{\dtt}\sum_{i=1}^{k}\mathbbm{1}_{N_{j}}(i)\upalpha_{\sigma_j}\\&
            =\sum_{i=1}^{k}|P_{i}|\upalpha_{\sigma_i}-\sum_{j=k+1}^{\dtt}|N_{j}|\upalpha_{\sigma_j}
        \end{align*}
        concluding the proof.
    \end{proof}

\subsection{Strategy of the proof}\label{subsec:2.2}
Let 
\[\iota:\ W\to \mathbb{R}^{2}\]
be the map which maps an element of the Weyl group $W$ to the corresponding point on the $(\psi_k,h_k)$-plane via
\[\iota(w)=\big(\psi_k(w),h_k(w)\big).\]
Our strategy is to study the convex hull 
of the image of $\iota$, namely
\[\mathcal{C}=\Hull(\iota(W)).\]
As a convex polygon, it is sufficient to study its boundary $\partial\mathcal{C}$.

Consider first the following two elements of $\mathcal{C}$. 
Let $\tau^{1}\in W$ be the identity permutation, and define $\tau^{\dtt}\in W$ by
\begin{equation}\label{eq:3.1n}
\tau^{\dtt}_i=\begin{cases}
    \dtt-k+i& i\leq k\\
    i-k&i>k
\end{cases}.
\end{equation}
We consider some properties of $\tau^{1}$ and $\tau^{\dtt}$. First, note that the first $k$ entries of $\att^{\sigma}$, where $\sigma\in W$, are the largest possible in the case where $\sigma=\tau^{1}$, and are the smallest when $\sigma=\tau^{\dtt}$. Then it is clear by the definition of $\psi_{k}$ that $\iota(\tau^{1})$ has the largest $\psi_k$ value amongst the points in $\mathcal{C}$, and similarly $\iota(\tau^{\dtt})$ has the smallest  $\psi_k$ value. 
It can further be shown that $\iota(\tau^{1})$ and $\iota(\tau^{\dtt})$ also have the largest and smallest $h_{k}$ value, respectively. We defer the proof of this statement to Lemma~\ref{lemma:iota(d)_is_minimal} (but note that at least for $\tau^{1}$ this claim is easy, since $h_{k}(\tau^{1})=h(G,\att)$).
Finally, note that the $\psi_k$ value is positive for $\iota(\tau^{1})$ and negative for $\iota(\tau^{\dtt})$.
Using these observations, Figure~\ref{fig:3.1} shows a schematic drawing of $\partial\mathcal{C}$.

\begin{figure}
\scalebox{0.7}{
  \begin{tikzpicture}[>=latex]
\centering
\begin{axis}[
  axis x line=center,
  axis y line=center,
  width={\linewidth},
  xtick={-3,-2,-1,0,1,2,3,4,5},
  ytick={0,4,8,12},
  xticklabels={},
  yticklabels={},
  xlabel={$\psi_k$},
  ylabel={$h_k$},
  label style={font=\Huge},
  xlabel style={right},
  ylabel style={above},
  xmin= -4,
  xmax=6,
  ymin=0,
  ymax=16]

  \addplot[black] [mark=*] coordinates {(-3,1) (-1,7) (0.5, 10) (2,11) (3,11.5) (5,12)};
  \addplot[black] [dashed] coordinates {(-3,1) (5,12)};  
  \addplot[black] [mark=*] coordinates {(-3,1) (-2,1.5) (1,4) (3,7) (5,12)};   
    \node[] at (axis cs: 5+0.55,12) {\LARGE $\iota(\tau^{1})$};    
    \node[] at (axis cs: -3-0.55,1) {\LARGE $\iota(\tau^{\dtt})$};  
\end{axis}
\end{tikzpicture}
}
\caption{A schematic drawing of $\partial\mathcal{C}$ on the $(\psi_k,h_k)$-plane.}\label{fig:3.1}
\end{figure}
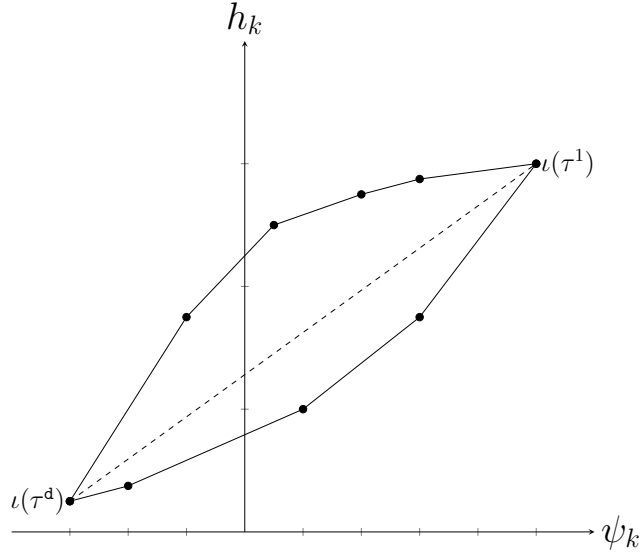

Note that for any two points on the boundary of a convex polygon, there are two components of the boundary which connect the two points. For the particular case of $\iota(\tau^{1})$ and $\iota(\tau^{\dtt})$, we claim that studying the upper component allows us to deduce bounds for entropy.

\begin{proposition}\label{proposition:2.1}
    Consider an edge in the upper component of $\partial\mathcal{C}$, which intersects the $\psi_k=0$ axis and has a slope $c$. Let $f$ be the $h_k$-value of this intersection point. 
    Let $\phi=c\cdot \psi_k$. Then
    \[\max_{w\in W}(h-\phi)([w]_{P_k})= f.\]
\end{proposition}
\begin{proof}
    For any $w\in W$, consider the line of slope $c$ which passes through $\iota(w)$.
    It can be observed that $(h-\phi)([w]_{P_k})$ is the $h_k$-value of the intersection point of said line with the $\psi_k=0$ axis. 
    For the particular case where $\iota(\tau)$ is a point on an edge of slope $c$ in the upper component of $\partial\mathcal{C}$, this line through $\iota(\tau)$ contains the edge and bounds $\mathcal{C}$ from above by convexity.
    Therefore, it intersects the $\psi_{k}=0$ axis at a higher point compared to the lines corresponding to other points $\iota(w)$. Namely, $h-\phi$ obtains its maximal value on $[\tau]_{P_k}$. As $f$ was defined as the $h_k$-value of this intersection point for $\tau$, we get $\underset{w\in W}{\max}(h-\phi)([w]_{P_k})= f,$ 
    as required.
\end{proof}
\begin{remark}\phantomsection\label{remark:3.3}
\begin{enumerate}
\item
    The proof of Proposition~\ref{proposition:2.1} can give an entropy bound derived from any edge of the upper component, not necessarily intersecting the $\psi_k=0$ axis.
    However, it is not difficult to see that the minimal bound is obtained from an edge connecting a point with a non-negative $\psi_k$ value and a point with a non-positive $\psi_k$ value, as in Proposition~\ref{proposition:2.1}.
\item\label{item:3.3_2}
    It is possible for the upper component of $\partial{C}$ to have two edges which intersect the $\psi_k=0$ axis, rather than just one. This happens if one of the vertices of the upper component lies precisely on the $\psi_{k}=0$ axis (which of course may occur only for some particular instances of $\att$). In this case both edges which contain said vertex indeed intersect the $\psi_{k}=0$ axis.
    Let $c_1<c_2$ be the slopes of the two edges.
    Then it is not difficult to see in this case that the statement of Proposition~\ref{proposition:2.1} holds with $\phi=c\psi_k$ for any $c_1\leq c\leq c_2$, using the same proof.
\item
Let us make a comment on the dynamical interpretation of the intersection point with the $\psi_{k}=0$ axis. This interpretation 
is given only as motivation and is not essential for any part of this current paper, and will be made clearer in~\cite{mor2022c} where it will be used to prove tightness of our entropy bounds.
As described in~\cite{mor2022a}, the cusp can be partitioned to regions according to parabolic subgroups, and points (lattices) in these regions can be associated with an orientation which is an element of $W_{P,\att}$.
To study entropy in the cusp, as also shown in~\cite{mor2022a}, one needs to understand trajectories of lattices which begin and end at the compact part of $G/\Gamma$ but spend most of the time in the cusp. 
Crudely, these trajectories contribute entropy $h(P,\att^{w})$ per unit time when passing through the cusp associated with $P$ at orientation $[w]_{P}$.
Considering the $(\psi_k,h_k)$-plane, note that for an edge between $\iota(\sigma)$ and $\iota(\tau)$ with
\[\psi_k(\sigma)>0>\psi_k(\tau),\]
the $h_k$ value of the intersection point of the edge with the $\psi_{k}=0$ axis is a weighted average of the entropies of the two orientations. Then this $h_{k}$ value is also equal to the entropy stemming from a trajectory in the $P_{k}$ part of the cusp, which goes up the cusp with orientation $[\tau]_{P_{k}}$ and down the cusp with orientation $[\sigma]_{P_{k}}$.
Similarly, for an edge which intersects the axis at a point $\iota(\sigma)$ for some $\sigma\in W$ with $\psi_k(\sigma)=0$, the entropy $h_k(\sigma)$ (which is the $h_{k}$ value of the intersection point)
can be interpreted as the contribution of a trajectory which stays high up in the cusp without changing its height, with orientation $[\sigma]_{P_k}$. 
\end{enumerate}
\end{remark}

Proposition~\ref{proposition:2.1} gives us bounds for the entropy of the cusp, as we desire.
\begin{corollary}\label{corollary:2.4}
    In the conditions of Proposition~\ref{proposition:2.1}, we have
    \[h_{\infty,\subseteq P_{k}}(\att)\leq f.\]
\end{corollary}
\begin{proof}
    Clearly, Theorem~\ref{theorem:1.0} asserts that
    $h_{\infty,P_{k}}(\att)\leq f$, using $\phi=c\cdot \psi_{k}$ and Proposition~\ref{proposition:2.1}.
    In order to see that the same bound holds for $h_{\infty,\subseteq P_{k}}(\att)$, let
    \[\tilde{\phi}=c\cdot\Big(\frac{1}{k}\sum_{i=1}^{k}\lambda_i-\frac{1}{\dtt-k}\sum_{i=k+1}^{\dtt}\lambda_i\Big)\]
    where $\lambda_i$ is defined as in~\S\ref{subsec:2.1}.
    Then
    \[\tilde{\phi}(\pi_{Q}(\alpha))=\phi(\pi_{P_{k}}(\alpha))\]
    for any $Q\subseteq P_{k}$ and $\alpha\in \Lie(A)$.
    As the entropy is monotone with respect to increasing parabolic subgroups, we also have
    \[h(Q,\alpha)\leq h(P_{k},\alpha),\]
    hence together
    \[(h-\tilde{\phi})([w]_{Q})\leq (h-\phi)([w]_{P_{k}})\]
    for any $w\in W$ and $Q\subseteq P_{k}$. This concludes the proof, using Theorem~\ref{theorem:1.0}.\end{proof}

\subsection{Simple transpositions}\label{sec:transpositions}
Let $T_n=(n, n+1)$ be the adjacent transposition of $n$ and $n+1$. 
Recall that $W\cong S_{\dtt}$ acts on $A$ from the right, so for any $\sigma\in W$ and $1\leq n\leq \dtt-1$, $\upalpha^{T_{n}\sigma}$ is obtained from $\upalpha^{\sigma}$ by switching between the entries of $\upalpha^{\sigma}$ with values $\upalpha_{n}$ and $\upalpha_{n+1}$.
Studying the effect of adjacent transpositions on entropy would be relevant for studying $\partial\mathcal{C}$. In fact, we will construct the upper component of $\partial\mathcal{C}$ by starting with $\iota(\tau^{1})$ and applying these transpositions, in a specific order, on the underlying elements of $W$.

Let us consider the effect of applying a simple transposition on both coordinates of $\iota(\sigma)$:
\begin{proposition}\label{calc_transposition}
Let $[\sigma]_{P_k}\in W_{P_k,\att}$ and $n<\dtt$. 
Assume $\sigma_{j}=n$ and $\sigma_{j^{\prime}}=n+1$ for some $j\leq k<j^{\prime}$. Then
\begin{enumerate}
    \item\label{item:8.4_1} $ h_k(\sigma)-h_k(T_n\sigma)=(j^{\prime}-j)(\upalpha_n-\upalpha_{n+1})$
    \item \label{item:8.4_2}
    $\psi_k(\sigma)-\psi_k(T_n\sigma)=\frac{\dtt}{k(\dtt-k)}(\upalpha_n-\upalpha_{n+1})$
\end{enumerate}
\end{proposition}

\begin{proof}
    Recall our convention in~\S\ref{subsec:2.1} for the representative $\sigma$ of $[\sigma]_{P_k}$, chosen so that $\sigma_{1}<\cdots <\sigma_{k}$ and $\sigma_{k+1}<\cdots<\sigma_{\dtt}$. 
    Then, by counting the entries of $\sigma$ equal to $1,\ldots,n-1$, it follows that
    \[(j-1)+(j^{\prime}-k-1)=n-1,\]
    hence
    \begin{equation}\label{eq:2.1}j^{\prime}=n+k-j+1.\end{equation}
    To prove item~\ref{item:8.4_1},
    we compute the entropy difference for the Levi subgroup $L_{P_{k}}$ and for the unipotent radical $U_{P_{k}}$ separately.
    First, using Lemma~\ref{lemma:levi_computation},
    \begin{align*}
        h(L_{P_{k}},\att^{\sigma})-h(L_{P_{k}},\att^{T_n \sigma})&=\Big((k+1-2j)\upalpha_n+(\dtt-k+1-2(j^{\prime}-k))\upalpha_{n+1}\Big)\\&
        -\Big((k+1-2j)\upalpha_{n+1}+(\dtt-k+1-2(j^{\prime}-k))\upalpha_n\Big)\\&
        =(2j^{\prime}-2j-\dtt)(\upalpha_n-\upalpha_{n+1}).
    \end{align*}
    Next, we use Lemma~\ref{lemma:3.3N} to compute the entropy difference in the unipotent radical. It follows that the difference is only in the coefficients of  $\upalpha_{n}$ and $\upalpha_{n+1}$, specifically:
    \begin{align*}
        h(U_{P_{k}},\att^{\sigma})-h(U_{P_{k}},\att^{T_n \sigma})&=\Big((\dtt-j^{\prime}+1)\upalpha_n-j \upalpha_{n+1}\Big)-\Big((\dtt-j^{\prime})\upalpha_{n+1}-(j-1)\upalpha_{n}\Big)\\&
        =(\dtt+j-j^{\prime})(\upalpha_n-\upalpha_{n+1}).
    \end{align*} 
    Together,
    \begin{align*}
         h_k(\sigma)-h_k(T_n\sigma)&=\Big(h(L_{P_{k}},\att^{\sigma})-h(L_{P_{k}},\att^{T_{n}\sigma})\Big)+\Big(h(U_{P_{k}},\att^{\sigma})-h(U_{P_{k}},\att^{T_{n}\sigma})\Big)\\&
        =(j^{\prime}-j)(\upalpha_n-\upalpha_{n+1})
    \end{align*} 
    which concludes the proof of item~\ref{item:8.4_1}.
    
    For item~\ref{item:8.4_2}, note that
    \begin{equation*}
    \psi_k(\sigma)-\psi_k(T_n\sigma)=\big(\frac{1}{k}\upalpha_n-\frac{1}{\dtt-k}\upalpha_{n+1}\big)-\big(\frac{1}{k}\upalpha_{n+1}-\frac{1}{\dtt-k}\upalpha_n\big)=\frac{\dtt}{k(\dtt-k)}(\upalpha_n-\upalpha_{n+1})
    \end{equation*}
    since the difference between the $\psi_k$ values is again only in the coefficients of $\upalpha_{n}$ and $\upalpha_{n+1}$.
\end{proof}

\subsection{The upper component of $\partial\mathcal{C}$}
In order to give a complete description of the upper component of $\partial\mathcal{C}$, we consider finite sequences in $W$ where each element is obtained from the previous one by a simple transposition (of the form in Proposition~\ref{calc_transposition}). 
Such a sequence is projected to the $(\psi_k,h_k)$-plane by $\iota$, and the points are connected consecutively by straight lines. Such a curve is then called a path.

Paths will be useful to work with since the difference between the coordinates of consecutive points is easily computed by Proposition~\ref{calc_transposition}. In particular, both the $\psi_{k}$ and $h_{k}$ values only decrease from one point to the next, forming a monotone non-decreasing curve.

\begin{definition}\label{definition:path}
    A \emph{path} for $[\sigma]_{P_k}\in W_{P_k,\att}$ is a piecewise linear curve in the $(\psi_k, h_k)$-plane obtained by 
    connecting points $\iota(w^{1}),\ldots,\iota(w^{s}) \in \mathcal C$ consecutively by straight line segments, with the following additional requirements:
    \begin{enumerate}
        \item $w^{1}=\tau^{1}$ and $w^{s}=\sigma$.
        \item
        For any $1\leq i<s$, there are some $1\leq n\leq \dtt-1$ and $j\leq k<j'$, so that
        $n=w^{i}_{j}$, $n+1=w^{i}_{j'}$, and $w^{i+1}=T_{n}w^{i}$.
    \end{enumerate}
\end{definition}

\begin{lemma}\label{lemma:2.6}
    There exists a path for any $[\sigma]_{P_k}\in W_{P_k,\att}$. 
\end{lemma}
\begin{proof}
The approach we take here is a useful way to think about the affect of applying simple transpositions, and will often be used throughout this paper. 
According to our convention, $\sigma$ is chosen so that $\sigma_1 < \sigma_2 < \dots < \sigma_k$ as well as $\sigma_{k+1} < \dots < \sigma_\dtt$. 
Then, in order to obtain $\sigma$ from $\tau^{1}$ we perform the following procedure. First we deal with the $k$th entry. We apply $T_{k}$ from the left to place $k+1$ in the $k$th place instead of $k$, apply $T_{k+1}$ next to place $k+2$ there and so on until we finally use $T_{\sigma_{k}-1}$ to place $\sigma_k$ in the $k$th entry. That is, we apply  $T_{\sigma_{k}-1} \cdots T_{k}$ to treat the $k$th entry, without changing the values of the first $k-1$ entries. Note that $\sigma_{k}\geq k$, so this product should be understood as the identity map if $\sigma_{k}=k$. Once this is done, there is a sufficient gap between the $(k-1)$th and $k$th entries (which are $k-1$ and $\sigma_{k}$, respectively) and we can similarly `push' the $(k-1)$th entry forward until we get it to be $\sigma_{k-1}$ using $T_{\sigma_{k-1}-1}\cdots T_{k-1}$ (this occurs again without affecting any other entry).
Eventually, it follows that we can write
    \[\sigma=(T_{\sigma_1-1}\cdots T_{1})\cdot(T_{\sigma_{2}-1}\cdots T_{2})\cdots (T_{\sigma_{k}-1}\cdots T_{k})\tau^{1}.\]
    Note that in this large product, each time a simple transposition $T_{n}$ multiplies from the left an element $w$, it happens so that $n=w_{j}$ and $n+1=w_{j'}$ for some $j\leq k<k'$, due to the gap created between the entries. Then it follows that this product defines a path for $[\sigma]_{P_k}$ by connecting the points $\iota(\tau^{1}),\iota(T_{k}\tau ^{1}),\cdots,\iota(T_{\sigma_{k}-1}\cdots T_{k}\tau^{1}),\ldots,\iota(\sigma)$.
\end{proof}

The following lemma was already mentioned in~\S\ref{subsec:2.2}, but now we have the tools to prove it efficiently.
\begin{lemma}\label{lemma:iota(d)_is_minimal}
    Let $\tau^{\dtt}\in W$ be as defined in Equation~\eqref{eq:3.1n}. Then $h_{k}(\tau^{\dtt})\leq h_{k}(\sigma)$ for all $\sigma\in W$.
\end{lemma}
\begin{proof}
    Let $\sigma\in W$. Without loss of generality, since $h_k(\sigma)$ depends only on $[\sigma]_{P_k}$ and not on $\sigma$, we may assume $\sigma$ satisfies the conventions of~\S\ref{subsec:2.1}.
    Then, similarly to the proof of Lemma~\ref{lemma:2.6}, it is clear that $\tau^{\dtt}$ can be obtained from $\sigma$ by applying simple transpositions.
    It follows from Proposition~\ref{calc_transposition} that each such transposition only reduces entropy, hence $h_{k}(\tau^{\dtt})\leq h_{k}(\sigma)$ as required. 
\end{proof}

The next lemma considers the possible transpositions which induce a given slope on the $(\psi_{k},h_{k})$-plane. The motivation for this question will be made clear right after the proof.
\begin{lemma}\label{lemma:3.10}
    Let $[\sigma]_{P_k}\in W_{P_k,\att}$ which satisfies $\sigma_j=n$, $\sigma_{j^{\prime}}=n+1$ for some $1\leq n\leq \dtt-1$ and $j\leq k<j^{\prime}$. Assume $\iota(\sigma)\not=\iota(T_{n}\sigma)$. 
    Then the slope of the line connecting $\iota(\sigma)$ and $\iota(T_n\sigma)$ has slope 
    \[\frac{k(\dtt-k)}{\dtt}s,\]
    for some $s$, if and only if  \[\sigma_{j}=s+2j-k-1.\]
    Furthermore, in this case we have
    $j\in J_{s}$, where
    \begin{equation}\label{eq:3.3N}J_{s}\coloneqq\{k+1-s,\ldots,\dtt-s\}\cap\{1,\ldots,k\}.\end{equation}
\end{lemma}
\begin{proof}
    Proposition~\ref{calc_transposition} states that the slope of the line connecting $\iota(\sigma)$ and $\iota(T_n\sigma)$ is equal to 
    \[\frac{k(\dtt-k)}{\dtt}(j^{\prime}-j).\]
    Note that Equation~\eqref{eq:2.1} states a relation between $j^{\prime},j$ and $n=\sigma_{j}$, namely
    \[j'=\sigma_{j}+k-j+1.\]
    Together we get that slope of the line is equal to
    \[\frac{k(\dtt-k)}{\dtt}(\sigma_{j}+k-2j+1).\]
    Hence indeed, the slope is equal to $\frac{k(\dtt-k)}{\dtt}s$ if and only if 
    \[\sigma_{j}=s+2j-k-1,\]
    as required.
    
    Next, note that due to our conventions in~\S\ref{subsec:2.1}, $\sigma$ satisfies \[\sigma_{1}<\cdots <\sigma_{j}<\cdots <\sigma_{k}.\] 
    Hence $\sigma_{j}\geq j$ (since there are $j-1$ distinct smaller positive integers) and also $\sigma_{j}\leq \dtt-(k-j+1)$ (since there are $k-j+1$ distinct larger integers bounded by $\dtt$; these are $\sigma_{j+1},\ldots,\sigma_{k}$ as well as $\sigma_{j'}$). 
    Hence
    \[s+2j-k-1=\sigma_{j}\in\{j,\ldots,\dtt-1-k+j\}.\]
    Therefore, it follows from this inclusion that
    \[j\in J_{s}=\{k+1-s,\ldots,\dtt-s\}\cap\{1,\ldots,k\}.\]
\end{proof}

\begin{definition}
    The \emph{amount of time a path spends at slope $c$} is defined as the sum of lengths of intervals on the $\psi_k$-axis where the slope of the line segments is $c$.
\end{definition}

    \begin{corollary}\label{cor:3.11}
        Let
        \[d_{s}=\frac{\dtt}{k(\dtt-k)}\sum_{n\in N_{s}}(\upalpha_{n}-\upalpha_{n+1}),\]       
        where
        $N_{s}=\{s+2j-k-1:\ j\in J_{s}\}.$
        Then the amount of time any path spends at slope $\frac{k(\dtt-k)}{\dtt}s$ is bounded by $d_{s}$, for all integer $1\leq s\leq \dtt-1$.
    \end{corollary}
    \begin{proof}
        Follows from Lemma~\ref{lemma:3.10}
        and Proposition~\ref{calc_transposition}.
    \end{proof}

\begin{definition}\label{def:3N}
    Consider some path. We say that the path is:
\begin{enumerate}
    \item 
     \emph{maximized up to slope} $\frac{k(\dtt-k)}{\dtt}s$, if the amount of time it spends at slope $\frac{k(\dtt-k)}{\dtt}s^{\prime}$ is equal to $d_{s^{\prime}}$ for all integer $1\leq s^{\prime}< s$.
    \item
    \emph{maximized}, if it is maximized up to the largest slope of the path.
    \item
    \emph{optimal}, if any path contained in it is maximized.
\end{enumerate}
\end{definition}

\begin{remark}
    The difference between a maximized and an optimal path is in the order of slopes. Considering the slopes from right to left in the $(\psi_k,h_k)$-plane, the slopes must be increasing for an optimal path (namely such path is concave) while there is no such restriction for a maximized path.
\end{remark}

Next, in Proposition~\ref{prop:3.12} we show that an optimal path for $[\tau^{\dtt}]_{P_k}$ must actually be equal to the upper component of $\partial\mathcal{C}$. Then our strategy for understanding this upper component would be to build explicitly an optimal path, which we do in Corollary~\ref{cor:3.18N}.
Once we do this, it would give a characterization of the upper component as the unique optimal path for $[\tau^{\dtt}]_{P_k}$
    \begin{proposition}\label{prop:3.12}
        If an optimal path for $[\tau^{\dtt}]_{P_k}$ exists, then this path is the upper component of $\partial\mathcal{C}$.
    \end{proposition}
    \begin{proof}
        Let $w\in W$. 
        Let $\mathcal{L}$ be the restriction of the optimal path for $[\tau^{\dtt}]_{P_k}$, to the         interval $[\psi_{k}(w),\psi_{k}(\tau^{1})]$ on the $\psi_{k}$-axis.
        By Lemma~\ref{lemma:2.6}, there is a path $\mathcal{L}^{\prime}$ for $[w]_{P_k}$.
        Recall that, by definition, a path is formed by connecting consecutive points each obtained from the previous one by a simple transposition of the form described in Proposition~\ref{calc_transposition}. In particular, it follows from Proposition~\ref{calc_transposition} that each line segment in a path has a positive slope in the $(\psi_{k},h_{k})$-plane.
        Then both $\mathcal{L}$ and $\mathcal{L}^{\prime}$ are the graphs of monotone increasing piecewise linear functions on $[\psi_k(w),\psi_k(\tau^{1})]$, which coincide at the point $\psi_k(\tau^{1})$.
        Since $\mathcal{L}$ is the restriction of an optimal path, the average slope of $\mathcal{L}$ is smaller than that of $\mathcal{L}^{\prime}$ (because, going from the right to the left, $\mathcal{L}$ exhausts the smallest possible slopes first), and so the left endpoint of $\mathcal{L}^{\prime}$, namely $\iota(w)$, lies below $\mathcal{L}$.
        It follows that the optimal path for $[\tau^{\dtt}]_{P_k}$ bounds from above all of the points in $\iota(W)$. 
        As the optimal path is on the other hand contained in $\mathcal{C}$, it is clear that it is precisely the upper component of $\partial\mathcal{C}$.
    \end{proof}

     Our goal now is to build an optimal path for $[\tau^{\dtt}]_{P_k}$. 
    For simplicity, we will assume that
    $k\leq \dtt/2$. This is certainly not an essential assumption, and the proof works for the other case as well, but it simplifies notations. Once a complete picture of this case is established, the case $k>\dtt/2$ will follow easily (see~\S\ref{subsec:8.10}).
    Note that for $k\leq\dtt/2$, we have
    \begin{equation}\label{eq:3.3nn}J_{s}=\begin{cases}
    \{k-s+1,\ldots,k\} & 1\leq s< k \\ 
    \{1,\ldots,k\} & k\leq s\leq \dtt-k \\
    \{1,\ldots,\dtt-s\} &  \dtt-k<s\leq \dtt-1
    \end{cases}\end{equation}
    where $J_s$ is as in Equation~\eqref{eq:3.3N}.

    We will now construct elements $\tau^{1},\ldots,\tau^{\dtt}$ in $W$, so that the curve which connects their images in the $(\psi_k,h_k)$-plane consecutively by straight lines would be an optimal path for $[\tau^{\dtt}]_{P_k}$. 
    These elements will be defined inductively, 
    where $\tau^{s+1}$ is obtained from $\tau^{s}$ by applying all possible simple transpositions $T_{\tau^{s}_{j}}$ for $j\leq k$ which are of the form discussed in Proposition~\ref{calc_transposition}.
    The goal would then be to show that the point $\iota(\tau^{s+1})$ is obtained from $\iota(\tau^{s})$ by traveling to the left on the $(\psi_{k},h_{k})$ at slope $s$ for the maximal possible amount of time $d_{s}$ (see Corollary~\ref{cor:3.11}), hence the path would be optimal.
    
    In order to define the points $\tau^{s}$ we first consider the following lemma.
    \begin{lemma}\label{lemma:3.14n}
        For $\sigma\in W$ and $1\leq k\leq \dtt-1$, define
        \[I_{\sigma}=\{1\;\leq j\leq k:\ \exists j'>k,\ \sigma_{j'}=\sigma_{j}+1\}\]
        and
        $M_{\sigma}=\{\sigma_{j}:\ j\in I_{\sigma}\}$. Then all elements of $\{T_{n}:\ n\in M_{\sigma}\}$ are commuting.
    \end{lemma}
    \begin{proof}
        Let $n\in M_{\sigma}$. First, since $n=\sigma_{j}$ for some $j\leq k$, it follows from the definition of $I_{\sigma}$ and $M_{\sigma}$ (and from $\sigma$ being a bijection) that $n-1\not\in M_{\sigma}$ (because it would require $n=(n-1)+1=\sigma_{j'}$ for $j'>k$, in contradiction).
        Similarly, since $n+1=\sigma_{j'}$ for some $j'>k$, it follows that $n+1\not\in M_{\sigma}$.
        Then the assertion follows from the fact that the only simple transpositions not commuting with $T_{n}$ are $T_{n-1}$ (if $n>1$) and $T_{n+1}$ (if $n<\dtt-1$).
    \end{proof}
    \begin{remark}\label{remark:3.15n}
        Let $[\sigma]_{P_{k}}\in W_{P_k,\att}$, so that the choice of the representative $\sigma$ follows the convention $\sigma_{1}<\cdots<\sigma_{k}$ as usual.
        Then the set $I_{\sigma}$ can be interpreted as follows. An integer $1\leq j<k$ satisfies $j\in I_{\sigma}$ if and only if $\sigma_{j+1}>\sigma_{j}+1$ (namely there is a non-trivial gap between $\sigma_{j}$ and $\sigma_{j+1}$), and $k\in I_{\sigma}$ if and only if $\sigma_{k}<\dtt$.
    \end{remark}

    We can now define the elements $\tau^{1},\ldots,\tau^{\dtt}$ as follows.
    \begin{definition}\label{definition:4N}
        Let $\tau^{1}$ be the identity, as before.
        Assuming $\tau^{s}$ is defined for some $1\leq s\leq \dtt-1$, let
        \begin{equation}\label{eq:3.4}
        \tau^{s+1}\coloneqq(\prod_{n\in M_{\tau^{s}}} T_{n})\tau^{s}.\end{equation}
    \end{definition}
    \begin{remark}
        Definition~\ref{definition:4N} is an abuse of notation for the case $s=\dtt-1$, since $\tau^{\dtt}$ was already defined in Equation~\eqref{eq:3.1n}. 
        We will show in Proposition~\ref{prop:3.16N} that this new definition of $\tau^{\dtt}$ coincides with our original definition.
    \end{remark}

The main properties we need to show the construction satisfies are given in the following proposition. During the proof, we will also write down explicit formulas (Equations~\eqref{eq:3.5N},\eqref{eq:3.6N},\eqref{eq:3.65N}) for the elements $\tau^{1},\ldots,\tau^{\dtt}$. 
    \begin{proposition}\label{prop:3.16N}
        Assume $k\leq \dtt/2$.
        Then:
        \begin{enumerate}
        \item 
        For any $1\leq s\leq \dtt-1$ we have $I_{\tau^{s}}=J_{s}$, where $I_{\tau^{s}}$ is as in Lemma~\ref{lemma:3.14n} and $J_{s}$ is as in Equation~\eqref{eq:3.3nn}.
        \item
        For any $1\leq s\leq \dtt -1$ and any $j\in J_{s}$ we have $\tau^{s}_{j}=s+2j-k-1$.
        \item
        The definition for $\tau^{\dtt}$ in Definition~\ref{definition:4N} coincides with Equation~\eqref{eq:3.1n}.
        \end{enumerate}
    \end{proposition}
    \begin{proof}
        We prove the proposition by constructing $\tau^{1},\ldots,\tau^{\dtt}$ explicitly. 
        Writing these elements down will also be useful later when we compute the bounds for the entropy in the cusp.
        As will be clear momentarily, it is easier to describe $I_{\tau^s}$ for three different regions of $s$  separately; (a) $1\leq s< k$, (b) $k\leq s\leq \dtt-k$, and (c)  $\dtt-k<s\leq \dtt-1$.
        Furthermore, we will write down $\tau^{s}_j$ only for $j\leq k$. Since $\tau^{s}$ is obtained from $\tau^{1}$ by simple transpositions, clearly $(\tau^{s}_{j})_{j=k+1}^{\dtt}$ is monotonically increasing, so determining the values for $j>k$ given the values for $j\leq k$ is not difficult. 
        The view point described in Remark~\ref{remark:3.15n} will be useful to keep in mind throughout this proof.
        
       Considering $\tau^{1}$ first, note that only the $k$th entry can be `pushed forward' by a simple transposition,  since $\tau^{1}_{j}=j$ for all $j\leq k$. Namely, we have $I_{\tau^1}=\{k\}$.
        Considering $\tau^{2}=T_{k}\tau^{1}$ next, since the $k$th entry was increased by $1$, a non-trivial gap of size $2$ was formed between the $(k-1)$th and $k$th entries of $\tau^{2}$, hence clearly $I_{\tau^2}=\{k-1,k\}$ (since $\tau^{2}_{j}=j$ for $j\leq k-1$ and $\tau^{2}_{k}=k+1$).
        This procedure continues the same, where for any $1\leq s\leq k-1$ the element $\tau^{s+1}$ is obtained from $\tau^{s}$ by applying simple transpositions on the last $s$ entries, namely \[I_{\tau^{s}}=\{k-s+1,\ldots,k\}\qquad\text{for $1\leq s\leq k-1$}.\]
        Note that we did indeed get $I_{\tau^{s}}=J_{s}$ for all such $s$
        (c.f.\  Equation~\eqref{eq:3.3nn}).
    
        Said differently, it follows that the $j$th entry is increased precisely \[\begin{cases} 0 & j\leq k-s \\ s-(k-j) & k-s<j\leq k \end{cases}\] times to form $\tau^{s+1}$ from $\tau^{1}$.
        Therefore, given that $\tau^{1}_{j}=j$ for all $1\leq j \leq k$, it follows that 
        \begin{equation}\label{eq:3.5N}\tau^{s}_j=\begin{cases} j & j\leq k-s+1 \\
    2j-k+s-1
    & k-s+1<j\leq k
    \end{cases}\qquad \text{for $1\leq s\leq k$}.\end{equation}
    Note that $\tau^{s}_{j}=2j-k+s-1$ for $j\in J_{s}$, as required (it is immediate for $j>k-s+1$, while for $j=k-s+1$ it follows since $2j-k+s-1=j$).

    \medskip

    For the next range of values for $s$, let us first write down $\tau^{k}$ explicitly. We have:
    \begin{equation}\label{eq:3.55N}\tau^{k}_{j}=2j-1\qquad\forall 1\leq j\leq k.\end{equation}
    Clearly, $I_{\tau^{k}}$ satisfies $I_{\tau^{k}}=\{1,\ldots,k\}$ due to the gap of size $2$ between $\tau^{k}_{j}$ and $\tau^{k}_{j+1}$ for all $1\leq j\leq k-1$.
    Proceeding inductively, the assertion $I_{\tau^{s}}=\{1,\ldots,k\}$ continues to hold until the $k$th entry finally reaches the value of $\dtt$ (at which point it cannot be further increased). Since $\tau^{k}_{k}=2k-1$, and the value increases by $1$ each time, we have \[I_{\tau^{s}}=\{1,\ldots,k\} \qquad\text{for  $k\leq s\leq \dtt-k$},\] which is indeed equal to $J_{s}$ for all such $s$.
    Then $\tau^{s+1}_{j}$ is obtained from $\tau^{k}$ by increasing all entries $s+1-k$ times, namely
    \begin{equation}\label{eq:3.6N}\tau^{s}_{j}=2j-1+s-k\qquad \forall 1\leq j\leq k\qquad \text{for $k+1\leq s\leq \dtt-k+1$},\end{equation}
    and indeed $\tau^{s}_{j}$ satisfies the required identity.
    
    \medskip

    Finally, for the third range of values of $s$, let us write $\tau^{\dtt-k+1}$ explicitly:
    \begin{equation}\label{eq:3.625N}\tau^{\dtt-k+1}_{j}=\dtt-2(k-j)\qquad \forall 1\leq j\leq k.\end{equation}
    At this point the $k$th entry reached its maximal value $\dtt$ so it cannot be pushed forward anymore, but all other entries can since the gap of size $2$ still remains between all entries. Then we have $I_{\tau^{\dtt-k+1}}=\{1,\ldots,k-1\}$. 
    It then follows that the $(k-1)$th entry $\tau^{\dtt-k+2}_{k-1}=\dtt-1$ also meets its maximal possible value, and in this case $I_{\tau^{\dtt-k+2}}=\{1,\ldots,k-2\}$. 
    In general, for any $\dtt-k<s\leq \dtt-1$ we would have \[I_{\tau^{s}}=\{1,\ldots,\dtt-s\},\]
    again satisfying $I_{\tau^{s}}=J_{s}$ as required.

    To conclude, let us write down $\tau^{s}$ for this region of $s$.
    We have that $\tau^{s+1}$ is obtained from $\tau^{s}$ by applying simple transpositions on the first $\dtt-s$ entries. That is, the $j$th entry is pushed forward precisely $\begin{cases}
        s-\dtt+k& j\leq \dtt-s\\
        k-j & \dtt-s< j\leq k
    \end{cases}$ times to form $\tau^{s+1}$ from $\tau^{\dtt-k+1}$, and we can write
    \begin{equation}\label{eq:3.65N}\tau^{s}_j=\begin{cases}
    2j-k+s-1 & j\leq\dtt-s+1\\
    \dtt-k+j &  \dtt-s+1 <  j\leq k
    \end{cases}\qquad \text{for $\dtt-k+2\leq s\leq \dtt$}.\end{equation}
    Indeed, $\tau^{s}_{j}=2j-k+s-1$ for any $j\in J_{s}$, for this last region of $s$, as required.
    
    Finally, note that both equations~\eqref{eq:3.65N} and~\eqref{eq:3.1n} give the same formula for $\tau^{\dtt}_{j}$ for $j\leq k$. This concludes the proof.
    \end{proof}

Let us now show that the properties in Proposition~\ref{prop:3.16N} do imply that the constructed path is optimal.
    \begin{corollary}\label{cor:3.18N}
        Assume $k\leq \dtt/2$. Then the curve connecting $\iota(\tau^{1}),\ldots,\iota(\tau^{\dtt})$ on the $(\psi_{k},h_{k})$-plane is an optimal path for $[\tau^{\dtt}]_{P_k}$.
    \end{corollary}
    \begin{proof}
    The curve is certainly a path for $[\tau^{\dtt}]_{P_k}$. We claim it is optimal.
        By Proposition~\ref{prop:3.16N} and  Lemma~\ref{lemma:3.10}, each simple transposition applied in Equation~\eqref{eq:3.4} induces a segment of slope $\frac{k(\dtt-k)}{\dtt}s$ on the $(\psi_k,h_{k})$-plane.
        Furthermore, by Proposition~\ref{calc_transposition}
        the difference between the $\psi_{k}$ coordinates of $\iota(\tau^{s})$ and $\iota(\tau^{s+1})$ is precisely \[\frac{\dtt}{k(\dtt-k)}\sum_{n\in M_{\tau^s}}(\upalpha_{n}-\upalpha_{n+1})=d_{s},\]
        where $d_{s}$ is as in Corollary~\ref{cor:3.11} and we used the fact that \[M_{\tau^{s}}=\{\tau^{s}_{j}:\ j\in I_{\tau^{s}}\}=\{s+2j-k-1:\ j\in J_{s}\}=N_{s}.\]

        Then indeed $\iota(\tau^{s+1})$ is obtained from $\iota(\tau^{s})$ by traveling to the left at slope $\frac{k(\dtt-k)}{\dtt}s$ for time $d_{s}$. By definition this means that the path is optimal (c.f.\ Definition~\ref{def:3N}).
    \end{proof}
    
    In conclusion, in Corollary~\ref{cor:3.18N} we gave full characterization of an optimal path for $[\tau^{\dtt}]_{P_k}$. By Proposition~\ref{prop:3.12} this path must be precisely the upper component of $\partial\mathcal{C}$. Hence we showed the following.
    \begin{corollary}\label{cor:3.14}
        The upper component of $\partial\mathcal{C}$ is the optimal path obtained by connecting the points $\iota(\tau^{1}),\ldots,\iota(\tau^{\dtt})$ consecutively by straight line segments, and the slope between two distinct points $\iota(\tau^{s})$ and $\iota(\tau^{s+1})$ is $\frac{k(\dtt-k)}{\dtt}s$.
    \end{corollary}

\subsection{Proof of Theorem~\ref{theorem:1.1} for the $k\leq \dtt/2$ case}\label{subsec:8.8}
In this section we prove the upper bound for $h_{\infty,\subseteq P_{k}}(\att)$ for the $k\leq \dtt/2$ case.
Previously we considered all elements $\tau^{1},\ldots,\tau^{\dtt}$ in order to give a full characterization of the upper component of $\partial\mathcal{C}$.
In view of Corollary~\ref{corollary:2.4}, we only need to determine which edge of this component intersects the $\psi_{k}=0$ axis, and compute the entropy bound derived by it.
First, we show that for this goal, it is sufficient to only consider $\tau^{k},\ldots,\tau^{\dtt-k+1}$.
\begin{lemma}\label{lemma:3.14}
    Let $(x_i)_{i=1}^{n}$ be a monotone non-increasing sequence with $\sum_{i=1}^{n}x_i=0$. Then for any $0\leq l\leq n$ we have $\sum_{i=1}^{l}x_i\geq 0\geq \sum_{i=l+1}^{n}x_i$.
\end{lemma}
\begin{proof}
    Assume, for the sake of contradiction, that \begin{equation}\label{eq:3.7N}\sum_{i=1}^{l}x_i<0.\end{equation} Then in particular $l\geq 1$ and by monotonicity
    \[lx_{l}\leq \sum_{i=1}^{l}x_{i}<0.\]
    Since $x_{l}<0$, it follows again by monotonicity that \begin{equation}\label{eq:3.8N}\sum_{i=l+1}^{n}x_{i}\leq0.\end{equation}
    Combining Equations~\eqref{eq:3.7N}-\eqref{eq:3.8N} we obtain
    \[\sum_{i=1}^{n}x_{i}<0,\]
    in contradiction. Therefore, $\sum_{i=1}^{l}x_{i}\geq 0$ as required.

    Finally, we also have \[\sum_{i=l+1}^{n}x_{i}=-\sum_{i=1}^{l}x_i\leq 0\]
    as desired. \end{proof}

    \begin{lemma}\label{lemma:3.15}
        For $k\leq \dtt/2$, we have
        \[\psi_{k}(\tau^{k})\geq0\geq\psi_{k}(\tau^{\dtt-k+1}).\]
    \end{lemma}
    \begin{proof}
        Using the explicit form for $\tau^{k}$ as in Equation~\eqref{eq:3.55N}, as well as the fact that $\sum_{i=1}^{\dtt}\upalpha_{i}=0$, we have
    \begin{equation}\psi_{k}(\tau^{k})=\frac{1}{k}\sum_{i=1}^{k}\upalpha_{2i-1}-\frac{1}{\dtt-k}(-\sum_{i=1}^{k}\upalpha_{2i-1})=\frac{\dtt}{k(\dtt-k)}\sum_{i=1}^{k}\upalpha_{2i-1}.\end{equation}
    Then, using monotonicity $\upalpha_{i}\geq \upalpha_{i+1}$ for all $1\leq i\leq \dtt-1$, we have
    \[2\sum_{i=1}^{k}\upalpha_{2i-1}\geq \sum_{i=1}^{k}(\upalpha_{2i-1}+\upalpha_{2i})=\sum_{i=1}^{2k}\upalpha_{i}\geq 0,\]
    where for the last inequality we used Lemma~\ref{lemma:3.14}.
    This shows that $\psi_{k}(\tau^{k})\geq 0$, as desired.

    Similarly, considering $\tau^{\dtt-k+1}$ (using Equation~\eqref{eq:3.625N}), we have
    \[\psi_{k}(\tau^{\dtt-k+1})=\frac{\dtt}{k(\dtt-k)}\sum_{i=1}^{k}\upalpha_{\dtt-2(k-i)}\]
    and
    \[2\sum_{i=1}^{k}\upalpha_{\dtt-2(k-i)}\leq \sum_{i=1}^{k}(\upalpha_{\dtt-2(k-i)}+\upalpha_{\dtt-2(k-i)-1})=\sum_{i=\dtt-2k+1}^{\dtt}\upalpha_{i}\leq 0,\]  
    concluding the proof. \end{proof}
    
We can finally prove Theorem~\ref{theorem:1.1} for the $k\leq \dtt/2$ case.
\begin{proof}
    
    Let $s$ be the minimal integer with $k\leq s\leq \dtt-k$ so that 
    \[\psi_{k}(\tau^{s})\geq 0\geq \psi_{k}(\tau^{s+1}).\]
    Such $s$ certainly exists since, due to Lemma~\ref{lemma:3.15}. 
    Therefore, it follows that the edge connecting $\iota(\tau^{s})$ and $\iota(\tau^{s+1})$ (alternatively a single point, if the two points are equal) intersects the $\psi_{k}=0$ axis. 
    Then it follows from Proposition~\ref{proposition:2.1} and Corollary~\ref{corollary:2.4}, combined with Corollary~\ref{cor:3.14} and Remark~\ref{remark:3.3} (item~\ref{item:3.3_2}), that  
\begin{equation}\label{eq:3.6}h_{\infty,\subseteq P_k}(\att)\leq (h-\phi_k)([\tau^{s}]_{P_k})\end{equation}
for 
$\phi_k=C_{k}\psi_{k}$, where $C_{k}=\frac{k(\dtt-k)}{\dtt}s$.
Let us compute the right hand side of Equation~\eqref{eq:3.6}, using the explicit form for $\tau^{s}$ as in Equations~\eqref{eq:3.55N} and~\eqref{eq:3.6N}. 
For simplicity of notation, let $m_k=s-k+1$. Then $\tau^{s}_{j}=m_{k}+2(j-1)$ for all $1\leq j\leq k$.
We first compute the entropy on the Levi subgroup, using Lemma~\ref{lemma:levi_computation}:    \begin{align*}
        h(L_{P_k},\att^{\tau^{s}})&=\sum_{i=1}^{k}(k+1-2i)\upalpha_{m_k+2(i-1)}
        +\sum_{i=1}^{m_k-1}(\dtt-k+1-2i)\upalpha_{i}\\&
        +\sum_{i=1}^{k}\Big(\dtt-k+1-2(m_k-1+i)\Big)\upalpha_{m_k+2(i-1)+1}\\&
        +\sum_{i=1}^{\dtt-m_k-2k+1}\Big(\dtt-k+1-2(m_k-1+k+i)\Big)\upalpha_{m_k+2k-1+i}.
    \end{align*}
    Let us now compute the entropy of $\att^{\tau^{s}}$ on the unipotent radical using the sets $P_{i}$ and $N_{j}$ as defined in Lemma~\ref{lemma:3.3N}.
    It is not difficult to write down the sets $P_{i}$ and $N_{j}$ explicitly and compute their sizes. For example, for any $1\leq i\leq k$ we have
    \[\{\tau^{s}_{j}:\ j\in P_{i}\}=\{m_{k}+2(l-1)+1:\ i\leq l\leq k\}\cup\{l:\ m_{k}+2k\leq l\leq \dtt\}\]
    so $|P_{i}|=(\dtt-m_k-2k+1)+(k-i+1)$.
    Then, after also computing $|N_{i}|$ in a similar way, we see that
    \begin{align*}
        h(U_{P_k},\att^{\tau^{s}})&=\sum_{i=1}^{k}\big((\dtt-m_k-2k+1)+(k-i+1)\big)\upalpha_{m_k+2(i-1)}\\&
        -\sum_{i=1}^{k}i\cdot \upalpha_{m_k+2(i-1)+1}
        -\sum_{i=1}^{\dtt-m_k-2k+1}k\cdot \upalpha_{m_k+2k-1+i}.
    \end{align*} 
Summing the entropy on the Levi subgroup and the unipotent radical, we obtain
    \begin{align*}
        h(P_k,\att^{\tau^{s}})&=h(L_{P_k},\att^{\tau^{s}})+h(U_{P_k},\att^{\tau^{s}})\\&
        =\sum_{i=1}^{m_k-1}(\dtt-k+1-2i)\upalpha_i
        +\sum_{i=1}^{k}\Big(\dtt-m_k-3(i-1)\Big)\upalpha_{m_k+2(i-1)}\\&
        +\sum_{i=1}^{k}\Big(\dtt-2m_k-k-3(i-1) \Big)\upalpha_{m_k+2(i-1)+1}\\&
        +\sum_{i=1}^{\dtt-m_k-2k+1}\Big(\dtt-4k-2m_k+1-2(i-1)\Big)\upalpha_{m_k+2k-1+i}\\&
        =\sum_{i=1}^{\dtt}(\dtt+1-2i)\upalpha_i
        -k\sum_{i=1}^{m_k-1}\upalpha_i\\&
        +\sum_{i=1}^{k}(m_k+i-2)\upalpha_{m_k+2(i-1)}
        +\sum_{i=1}^{k}(i-k)\upalpha_{m_k+2(i-1)+1}.
    \end{align*}
To compute $\phi_k([\tau^{s}]_{P_k})$, note that $\sum_{i=1}^{\dtt}\upalpha_{i}=0$ and so
\[\psi_k(\tau^{s})=\frac{1}{k}\sum_{i=1}^{k}\upalpha_{m_k+2(i-1)}-\frac{1}{\dtt-k}(-\sum_{i=1}^{k}\upalpha_{m_k+2(i-1)})=\frac{\dtt}{k(\dtt-k)}\sum_{i=1}^{k}\upalpha_{m_k+2(i-1)}.\]
    
Putting it all together,
    \begin{align*}
        &(h-\phi_k)([\tau^{s}]_{P_k})=h(P_k,\att^{\tau^{s}})-\frac{k(\dtt-k)}{\dtt}(m_k+k-1) \psi_k(\tau^{s})\\&
        =\sum_{i=1}^{\dtt}(\dtt+1-2i)\upalpha_i
        -k\sum_{i=1}^{m_k-1}\upalpha_i
        +\sum_{i=1}^{k}(i-k-1)\upalpha_{m_k+2(i-1)}
        +\sum_{i=1}^{k}(i-k)\cdot \upalpha_{m_k+2(i-1)+1}\\&
        =h(G,\att)-k\sum_{i=1}^{m_k}\upalpha_{i}-\sum_{i=1}^{k-1}(k-i)(\upalpha_{m_k+2i-1}+\upalpha_{m_k+2i}).
    \end{align*}
This gives the desired entropy bound. 
To conclude the proof, we only need to show that the constant $m_{k}$ defined in the proof agrees with the definition in Theorem~\ref{theorem:1.1}.
Recall that $s$ was defined as the minimal integer $k\leq s\leq \dtt-k$ so that $\psi_{k}(\tau^{s})\geq 0\geq \psi_{k}(\tau^{s+1})$.
Using the definition $m_{k}=s-k+1$, this characterizes $m_{k}$ as the minimal integer $1\leq m\leq \dtt-2k+1$ so that
\begin{equation}\label{eq:3.90n}\psi_{k}(\tau^{m+k-1})\geq 0\geq \psi_{k}(\tau^{m+k}).\end{equation}
Since $k\leq m+k-1<m+k\leq \dtt-k+1$, we can use Equations~\eqref{eq:3.55N} and~\eqref{eq:3.6N} to deduce 
\[\psi_{k}(\tau^{m+k-1})=\frac{\dtt}{k(\dtt-k)}\sum_{i=1}^{k}\upalpha_{m+2(i-1)}\]
and
\[\psi_{k}(\tau^{m+k})=\frac{\dtt}{k(\dtt-k)}\sum_{i=1}^{k}\upalpha_{m+2(i-1)+1}\]
which shows that the definition of $m_k$ in Equation~\eqref{eq:3.90n} coincides with the definition in Theorem~\ref{theorem:1.1}.
This concludes the proof.
\end{proof}

\begin{remark}\label{remark:2.13}
    It follows from the proof above, using Remark~\ref{remark:3.3} (item~\ref{item:3.3_2}), that Theorem~\ref{theorem:1.1} also holds if $m_{k}$ is replaced by any $m$, not necessarily minimal, satisfying 
    \[\sum_{i=1}^{k}\upalpha_{m+2(i-1)}\geq 0\geq\sum_{i=1}^{k}\upalpha_{m+2(i-1)+1}.\]
    In this case, the linear functional $\phi=\frac{k(\dtt-k)}{\dtt}(m+k-1)\psi_{k}$ could be used to deduce the same upper bound for entropy.
\end{remark}

\subsection{Proof of Theorem~\ref{theorem:1.1} for the $k>\dtt/2$ case}\label{subsec:8.10}
We can now use the $k\leq \dtt/2$ case to complete the proof of Theorem~\ref{theorem:1.1}.
\begin{proof}
Let $P=P_{k}$ for $k>\dtt/2$. Let
$Q=P_{\dtt-k}$ be the parabolic subgroup defined by the same block sizes as $P$, but with reverse order.  
For simplicity let $\tilde{k}=\dtt-k$.
For $w\in W$, let $\tilde{w}\in W$ be defined by
\[\tilde{w}_j=\begin{cases}w_{j+k} & \textup{for }1\leq j\leq \dtt-k\\w_{j-(\dtt-k)} & \textup{for }\dtt-k<j\leq \dtt\end{cases},\]
i.e.\ the permutation obtained by changing the order between the first $k$ entries and the last $\dtt-k$ entries. 
Note that 
\[\psi_{k}(w)=-\psi_{\tilde{k}}(\tilde{w}).\]
Furthermore,
$h(L_{P},\att^{w})=h(L_{Q},\att^{\tilde{w}})$ and $h(U_{P},\att^{w})=h(U_{Q}^{T},\att^{\tilde{w}})$,
where the $T$ superscript stands for transpose, since the same Lyapunov exponents are summed in both sides of the equalities.
Therefore,
\begin{align*}h(P,\att^{w})&=h(L_{P},\att^{w})+h(U_{P},\att^{w})=h(L_{Q},\att^{\tilde{w}})+h(U_{Q}^{T},\att^{\tilde{w}})\\&
=h(Q,\att^{\tilde{w}})+h(U_{Q}^{T},\att^{\tilde{w}})-h(U_{Q},\att^{\tilde{w}}).
\end{align*}
Also note that
\begin{align*}
    h(U_{Q}^{T},\att^{\tilde{w}})-h(U_{Q},\att^{\tilde{w}})&=\sum_{i=1}^{\tilde{k}}\sum_{j=\tilde{k}+1}^{\dtt}\Big((\upalpha_{\tilde{w}_j}-\upalpha_{\tilde{w}_i})^{+}-(\upalpha_{\tilde{w}_i}-\upalpha_{\tilde{w}_j})^{+}\Big)\\&
    =-\sum_{i=1}^{\tilde{k}}\sum_{j=\tilde{k}+1}^{\dtt}(\upalpha_{\tilde{w}_i}-\upalpha_{\tilde{w}_j})=-\Big((\dtt-\tilde{k})\sum_{i=1}^{\tilde{k}}\upalpha_{\tilde{w}_i}-\tilde{k}\sum_{j=\tilde{k}+1}^{\dtt}\upalpha_{\tilde{w}_j}\Big)\\&
    =-k(\dtt-k)\psi_{\tilde{k}}(\tilde{w}).
\end{align*}

Let $\phi_{\tilde{k}}=C_{\tilde{k}}\psi_{\tilde{k}}$ be the linear functional used in Equation~\eqref{eq:3.6} (note that $\tilde{k}<\dtt/2$), and define $\phi_{k}=C_{k}\psi_{k}$ where
\[C_{k}=k(\dtt-k)-C_{\tilde{k}}.\]
Then putting all together, we obtain
\begin{align*}
    (h-\phi_k)([w]_P)&=h(Q,\att^{\tilde{w}})-k(\dtt-k)\psi_{\tilde{k}}(\tilde{w})-C_{k}\psi_{k}(w)\\&
    =h(Q,\att^{\tilde{w}})-\Big(k(\dtt-k)-C_k\Big)\psi_{\tilde{k}}(\tilde{w})
    =(h-\phi_{\tilde{k}})([\tilde{w}]_{P_{\tilde{k}}}).
\end{align*}
Taking the maximum over all $w\in W$, as well as applying the argument in the proof of Corollary~\ref{corollary:2.4} to get a bound for $h_{\infty,\subseteq P_k}(\att)$, we get
\begin{equation}\label{eq:3.8}h_{\infty,\subseteq P_{k}}(\att)\leq \max_{w\in W}(h-\phi_{k})([w]_{P_{k}})=\max_{\tau\in W}(h-\phi_{\tilde{k}})([\tau]_{P_{\tilde{k}}}).\end{equation}
As the term on the right hand side of Equation~\eqref{eq:3.8} is precisely the bound obtained in~\S\ref{subsec:8.8} for $h_{\infty,\subseteq P_{\tilde{k}}}(\att)$, this concludes the proof.
\end{proof}

\section{Bounding the entropy of the entire cusp}\label{section:4}
In this section we prove Theorem~\ref{theorem:1.3}, namely find an upper bound for the entropy of the cusp $h_{\infty}(\att)$, without restrictions on the parabolic subgroups involved.
We prove that the quantity $h(G,\att)-\sum_{i=1}^{\dtt}\upalpha_{i}^{+}$ which was shown in~\S\ref{section:h_for_pmax} to bound $h_{\infty,\subseteq P_k}(\att)$ for $k=1,\dtt-1$, also bounds the a-priori larger quantity $h_{\infty}(\att)$. As we mentioned, all of these bounds will be shown in~\cite{mor2022c} to be tight.
Note that if $\dtt=2$ there is only one proper standard parabolic subgroup of $G$, so Theorem~\ref{theorem:1.3} in fact follows immediately from Theorem~\ref{theorem:1.0} using either Theorem~\ref{theorem:1.1} or Theorem~\ref{theorem:1.4}. Therefore, we assume $\dtt>2$ throughout this section. 

In order to bound $h_{\infty}(\att)$ using Theorem~\ref{theorem:1.0}, we have to consider all parabolic subgroups (maximal or otherwise) simultaneously by the same linear functional, in contrast to our study
in~\S\ref{section:h_for_pmax} where we considered each of the maximal parabolic subgroups separately using a specialized linear functional. 
We discussed two approaches for a maximal parabolic subgroup $P_{k}$.
First, we used a linear functional $\phi_k$ proportional to $\psi_{k}$ to get a bound for $h_{\infty,P_k}(\att)$ 
(following the idea of
Proposition~\ref{proposition:2.1}, this bound was derived explicitly from a complete characterization of the upper component of $\partial\mathcal{C}$).
It is clear that the sum of these functionals over all $k$ gives a linear functional which restores the same optimal upper bound for $h_{\infty,P_k}(\att)$ for all $k$ simultaneously, but its behaviour for non-maximal parabolic subgroups is a-priorily unclear.
Alternatively, we used a functional proportional to $\frac{1}{k}\sum_{i=1}^{k}\lambda_i-\frac{1}{\dtt-k}\sum_{i=k+1}^{\dtt}\lambda_i$ to get a bound for the entropy $h_{\infty,\subseteq P_k}(\att)$ which also takes into account parabolic subgroups contained in $P_k$ (see the proof of Corollary~\ref{corollary:2.4}). 
In this case, however, summing the linear functionals would alter the upper bounds for entropy on the maximal parabolic subgroups and with it give a worse bound for $h_{\infty}(\att)$.

We take here a different approach, and define in~\S\ref{subsec:4.1} a new functional $\phi_{\all}$ which would be easier to study from a technical view point.
It would be fairly straightforward to show in~\S\ref{subsec:4.2} that $h-\phi_{\all}$ is larger on the maximal parabolic subgroups compared to non-maximal subgroups.  
It would also be clear that $\phi_{\all}$ identifies with $\phi_{1}$ and $\phi_{\dtt-1}$ on $P_{1}$ and $P_{\dtt-1}$, respectively, and so gives the desired bound for the entropy on these groups. However, it would not be the case for $k\not=1,\dtt-1$. 
Hence the main task, in~\S\ref{subsec:4.3}, would be to show that the largest entropy bound is still obtained on $P_{1}$ and $P_{\dtt-1}$ amongst the maximal parabolic subgroups.

\subsection{Defining $\phi_{\all}$}\label{subsec:4.1}
Recall that, as discussed in~\S\ref{subsec:2.1}, we assume the element $\upalpha=\diag(\upalpha_1,\ldots,\upalpha_{\dtt})\in \Lie(A)$ defining $\att=\exp(\upalpha)$ satisfies $\upalpha_{i}\geq \upalpha_{i+1}$ for all $1\leq i\leq \dtt-1$.
Further recall the definition of the functional $\lambda_i$ (for $1\leq i\leq \dtt$) as in~\S\ref{subsec:2.1}, namely $\lambda_{i}(\alpha)=\alpha_i$ for any $\alpha=\diag(\alpha_1,\ldots,\alpha_{\dtt})\in \Lie(A)$, and the functional $\psi_{i}=\lambda_{i}-\lambda_{i+1}$ (for $1\leq i\leq \dtt-1$).

Fix an integer $m$ so that $\upalpha_{m}\geq 0\geq \upalpha_{m+1}$, and let
\[\begin{cases}z_1=\frac{\dtt-1}{\dtt-2}(m-1) \\ z_{\dtt}=-\frac{\dtt-1}{\dtt-2}(\dtt-m-1)\end{cases}.\]
Then, we define
\[\phi_{\all}^{\att,m}=z_1\lambda_1+z_{\dtt}\lambda_{\dtt}.\]
We will often use the implicit notation $\phi_{\all}$ instead of $\phi_{\all}^{\att,m}$.

\begin{lemma}\label{lemma:new.1}
    For any $1\leq k\leq \dtt-1$, we have
    \[\phi_{\all}|_{\Lie(A_{P_k})}=\frac{(\dtt-1)\Big((\dtt-2k)m+\dtt(k-1)\Big)}{\dtt(\dtt-2)}\psi_{k}|_{\Lie(A_{P_k})}\]
\end{lemma}

\begin{proof}
    Note that for any $\alpha\in\Lie(A_{P_k})$ and any $1\leq k\leq \dtt-1$, we have
    \[\psi_k(\alpha)=\frac{\dtt}{\dtt-k}\lambda_1(\alpha)=-\frac{\dtt}{k}\lambda_{\dtt}(\alpha).\]
    Therefore, for any such $\alpha$, 
    \begin{align*}\phi_{\all}(\alpha)&=(\frac{\dtt-k}{\dtt}z_1-\frac{k}{\dtt}z_{\dtt})\psi_{k}(\alpha)=\frac{\dtt-1}{\dtt(\dtt-2)}\Big((\dtt-k)(m-1)+k(\dtt-m-1)\Big)\psi_{k}(\alpha)\\&
    =\frac{(\dtt-1)\Big((\dtt-2k)m+\dtt(k-1)\Big)}{\dtt(\dtt-2)}\psi_{k}(\alpha),
    \end{align*}
    as required.
\end{proof}

We show that this new functional $\phi_{\all}$ restores the same upper bound for the entropy of the
$P_1$ and $P_{\dtt-1}$ maximal parabolic subgroups as we had before.
\begin{corollary}\label{corollary:new.2}
    Let $P=P_{k}$ for $k=1$ or $k=\dtt-1$. Then
    \[\max_{[w]_{P}\in W_{P,\att}}(h-\phi_{\all})([w]_P)\leq h(G,\att)-\sum_{i=1}^{\dtt}\upalpha_{i}^{+}.\]
\end{corollary}
\begin{proof}
    First, by substituting $k=1$ in the result of Lemma~\ref{lemma:new.1} we have
    \[\phi_{\all}|_{\Lie(A_{P_1})}=\frac{\dtt-1}{\dtt}m\cdot\psi_{1}|_{\Lie(A_{P_1})}.\]
    Note that the linear functional on the RHS is (the restriction of) the one used in the proof of Theorem~\ref{theorem:1.1} to bound $h_{\infty,P_1}(\att)$ (see Equation~\eqref{eq:3.6} and the line below it, as well as Remark~\ref{remark:2.13}).
    Hence the same upper bound  $h(G,\att)-\sum_{i=1}^{\dtt}\upalpha_{i}^{+}$ is also obtained by using $\phi_{\all}$, as they identify on $\Lie(A_{P_1})$.
    
    The same argument applies similarly to the $k=\dtt-1$ case.
\end{proof}

We will later require the following properties of the proportionality constant characterizing $\phi_{\all}|_{\Lie(A_{P_k})}$ in Lemma~\ref{lemma:new.1}.
\begin{lemma}\label{lemma:new.6}
    Let $2\leq k\leq \dtt/2$ and set 
    \[c=\frac{(\dtt-1)\Big((\dtt-2k)m+\dtt(k-1)\Big)}{k(\dtt-k)(\dtt-2)}.\]
    Then the following hold:
    \begin{enumerate}[label=(\roman*)]
        \item\label{item-1} $c\geq \frac{\dtt-1}{\dtt-k}$
        \item $c\leq \frac{\dtt-1}{k}$
        \item $k+c<\dtt$
        \item $kc\geq m$.
    \end{enumerate}
\end{lemma}
\begin{proof}
    Note that as $\dtt-2k\geq 0$ and $m\geq 1$, we can deduce
    \[c\geq \frac{(\dtt-1)\Big((\dtt-2k)+\dtt (k-1)\Big)}{k(\dtt-k)(\dtt-2)}=\frac{\dtt-1}{\dtt-k},\]
    establishing \ref{item-1}.

    Next, similarly, as $m\leq \dtt-1$, we can deduce
    \[c\leq \frac{(\dtt-1)\Big((\dtt-2k)\cdot(\dtt-1)+\dtt( k-1)\Big)}{k(\dtt-k)(\dtt-2)}=\frac{(\dtt-1)(\dtt^{2}-2\dtt-k\dtt+2k)}{k(\dtt-k)(\dtt-2)}=\frac{\dtt-1}{k}.\]
    Then, as $k\in (1,\dtt-1)$, we also have the corollary
    \[c\leq \frac{\dtt-1}{k}< \dtt-k.\]

    Lastly, 
    \begin{align*}(\dtt-k)(\dtt-2)(kc-m)&=(\dtt-1)\Big((\dtt-2k)m+\dtt(k-1)\Big)-m(\dtt-k)(\dtt-2)\\&
    =\Big((\dtt-1)(\dtt-2k)-(\dtt-k)(\dtt-2)\Big)m+(\dtt-1)\dtt(k-1)\\&
    =-\dtt(k-1)m+\dtt(\dtt-1)(k-1)\\&
    =\dtt(k-1)(\dtt-m-1)\geq 0,
    \end{align*}
    concluding the proof.
\end{proof}

\subsection{Bounding $h-\phi_{\all}$ on non-maximal parabolic subgroups}\label{subsec:4.2}
The main reason why we chose to use the functional $\phi_{\all}$ is due to its nice behaviour with respect to the operation of increasing the parabolic subgroup. This is key for us as we try to bound the entropy in the cusp and hence need to consider all parabolic subgroups together.
This property is shown in the following proposition.

\begin{proposition}
    For any $Q\in\mathcal{P}$ there is some maximal parabolic subgroup $P$ so that $Q\subseteq P$ and 
    \[(h-\phi_{\all})([w]_{Q})\leq (h-\phi_{\all})([w]_{P})\]
    for all $w\in W$.
\end{proposition}
\begin{proof}
    First of all, we may assume that $Q$ has only three blocks. Indeed, if otherwise, we may define $Q^{\prime}$ as the group with only three blocks, the first and the last of which are of the same size as the first and last blocks of $Q$, respectively. That is, we unite all of the middle blocks together. Then, note that
    \[\phi_{\all}(\pi_{Q}(\alpha))=\phi_{\all}(\pi_{Q^{\prime}}(\alpha))\]
    for any $\alpha\in\Lie(A)$,
    as $\phi_{\all}$ only depends on the first and last diagonal entries of its argument, and these do not change by replacing $Q$ with $Q^{\prime}$.
    Then, as entropy is monotone with respect to increasing groups, clearly
    \[(h-\phi_{\all})([w]_{Q})\leq (h-\phi_{\all})([w]_{Q^{\prime}})\]
    for any $w\in W$.

    So now assume $Q$ is a parabolic subgroup with exactly three blocks, of sizes $k_1, k_2, k_3$, respectively. 
    We separate to two cases. First, assume $k_1\leq k_3$. 
    Let $P=P_{k_1}$ be the maximal parabolic with only two blocks, the first of which of size $k_1$.
    Let $\beta=\upalpha^{w}$.
    Then 
    \begin{align*}\phi_{\all}(\pi_{Q}(\beta))-\phi_{\all}(\pi_{P}(\beta))&=z_{\dtt}\Big(\frac{1}{k_3}\sum_{i=k_1+k_2+1}^{\dtt}\beta_{i}-\frac{1}{k_2+k_3}\sum_{i=k_1+1}^{\dtt}\beta_i\Big)\\&
    =z_{\dtt}\Big(\frac{k_2}{k_3(k_2+k_3)}\sum_{i=k_1+k_2+1}^{\dtt}\beta_i-\frac{1}{k_2+k_3}\sum_{i=k_1+1}^{k_1+k_2}\beta_i\Big)\\&
    =\frac{1}{k_3(k_2+k_3)}z_{\dtt}\Big(k_2\sum_{i=k_1+k_2+1}^{\dtt}\beta_i-k_3\sum_{i=k_1+1}^{k_1+k_2}\beta_i\Big)\\&
    =\frac{1}{k_3(k_2+k_3)}z_{\dtt}\cdot\!\sum_{i=k_1+k_2+1}^{\dtt}\,\sum_{j=k_1+1}^{k_1+k_2}(\beta_i-\beta_j).
    \end{align*}
    Therefore, we obtain that
    \begin{equation}\label{eq:3.1}(h-\phi_{\all})([w]_{P})-(h-\phi_{\all})([w]_{Q})=\sum_{i=k_{1}+k_{2}+1}^{\dtt}\,\sum_{j=k_1+1}^{k_1+k_2}\Big((\beta_i-\beta_j)^{+}+\frac{1}{k_3(k_2+k_3)}z_{\dtt}(\beta_i-\beta_j)\Big).\end{equation}

    Note that for any $x\in \mathbb{R}$ and $\kappa\in [-1,0]$, we have $x^{+}+\kappa x\geq 0$. Therefore, in order to show that the RHS in Equation~\eqref{eq:3.1} is non-negative, it is sufficient to show that
    \[\frac{1}{k_3(k_2+k_3)}z_{\dtt}\in [-1,0].\]
    Indeed, first of all \[\frac{1}{k_3(k_2+k_3)}z_{\dtt}\leq 0\]
    because $z_{\dtt}\leq 0$.
    Secondly, as $z_{\dtt}\geq -(\dtt-1)$ and $k_{3}\geq k_1$, 
    \[\frac{1}{k_3(k_2+k_3)}z_{\dtt}\geq -\frac{\dtt-1}{k_{1}(\dtt-k_1)}=-\Big(1-\frac{(k_1-1)(\dtt-k_1-1)}{k_1(\dtt-k_1)}\Big)\geq -1.\]

    Similarly, if $k_1> k_3$ instead, we let $P=P_{k_1+k_2}$ be the maximal parabolic subgroup whose second block is of size $k_3$, and find that
    \begin{align*}\phi_{\all}(\pi_{Q}(\beta))-\phi_{\all}(\pi_{P}(\beta))&=z_{1}\Big(\frac{1}{k_1}\sum_{i=1}^{k_1}\beta_i-\frac{1}{k_1+k_2}\sum_{i=1}^{k_1+k_2}\beta_i\Big)\\&
    =-\frac{1}{k_1(k_1+k_2)}z_1\Big(\sum_{i=k_1+1}^{k_1+k_2}\sum_{j=1}^{k_1}(\beta_i-\beta_j)\Big).
    \end{align*}
    Therefore, it is sufficient in this case to show that
    \[\frac{1}{k_1(k_1+k_2)}z_{1}\in [0,1].\]
    Indeed, on the one hand it is clear that this term is positive since $z_1\geq 0$, and on the other, since $k_1>k_3$ and $z_{1}\leq \dtt-1$, we have
    \[\frac{1}{k_1(k_1+k_2)}z_1\leq \frac{\dtt-1}{k_3(\dtt-k_3)}\leq 1,\]
    as required.
    This concludes the proof.
\end{proof}

\subsection{Bounding $h-\phi_{\all}$ on maximal parabolic subgroups}\label{subsec:4.3}
Now we are only left with studying the value of $h-\phi_{\all}$ on the maximal parabolic subgroups, and showing that the bound
$h(G,\att)-\sum_{i=1}^{\dtt}\upalpha_i^{+}$ which was computed for the entropy in the cusp for the parabolic groups $P_1$ and $P_{\dtt-1}$ is still a bound over all maximal parabolic subgroups.
We need to show the following.
\begin{proposition}\label{proposition:3.5}
    For
    any $\att=\exp(\upalpha)\in A$ and any $m$ such that $\upalpha_{m}\geq 0\geq \upalpha_{m+1}$, the inequality
    \[\max_{[w]\in W_{P,\att}}(h-\phi_{\all}^{\att,m})([w]_{P})\leq h(G,\att)-\sum_{i=1}^{\dtt}\upalpha_i^{+}\]
    holds for $P=P_k$ for all $1\leq k\leq \dtt-1$.
\end{proposition}

We first reduce the question to the case $k\leq \dtt/2$.
\begin{proposition}\label{proposition:3.6}
    Assume Proposition~\ref{proposition:3.5} holds for all $1\leq k\leq \dtt/2$. Then it holds for all $\dtt/2< k\leq \dtt-1$ as well.
\end{proposition}

\begin{proof}
    We show that the entropy bound obtained for $\att$ on $P_{k}$ can also be obtained as an entropy bound for $\att^{-1}$ on $P_{\dtt-k}$. 

    Let $\sigma\in W$ be defined by $\sigma_i=\dtt-i$, and define $\upbeta=(-\upalpha)^{\sigma}$. Note that the elements on the diagonal of $\mathtt{b}\coloneqq\exp(\upbeta)$ are ordered from largest to smallest as was for $\upalpha$. 
    Let $\tilde{m}=\dtt-m$ and $\tilde{k}=\dtt-k$. 
    Let $c$ be the constant defined in Lemma~\ref{lemma:new.6} for $k$ and $m$, and $\tilde{c}$ the constant defined for $\tilde{k}$ and $\tilde{m}$.
    Then, note that
    \[\tilde{c}=\frac{(\dtt-1)\Big((2k-\dtt)(\dtt-m)+\dtt(\dtt-k-1)\Big)}{k(\dtt-k)(\dtt-2)}=\frac{(\dtt-1)\Big((\dtt-2k)m+\dtt(k-1)\Big)}{k(\dtt-k)(\dtt-2)}=c.\]

    For 
    $w\in W$, let $\tilde{w}\in W$
    be the element obtained by switching the first $k$ entries and the last $\dtt-k$ entries of $w$, as in~\S\ref{subsec:8.10}.
    Define $\tau=\sigma^{-1}\tilde{w}$, so that $\beta^{\tau}=-\upalpha^{\tilde{w}}$.
    Then, observe that
    \[\psi_{\tilde{k}}(\pi_{P_{\tilde{k}}}(\upbeta^{\tau}))=\frac{1}{\dtt-k}\sum_{j=k+1}^{\dtt}(-\upalpha_{w_j})-\frac{1}{k}\sum_{i=1}^{k}(-\upalpha_{w_i})=\psi_{k}(\pi_{P_k}(\upalpha^{w})).\]
    Furthermore, note that $h(G,\mathtt{b})=h(G,\att)$, hence
    \[h(P_{\tilde{k}},\mathtt{b}^{\tau})=h(G,\mathtt{b})-\sum_{i=1}^{k}\sum_{j=k+1}^{\dtt}\big((-\upalpha_{w_i})-(-\upalpha_{w_j})\big)^{+}=h(P_{k},\att^{w}).\]
    All together, using Lemma~\ref{lemma:new.1}, we see that
    \begin{equation}\label{eq:3.2} h(P_k,\att^{w})-\phi_{\all}^{\att,m}(\pi_{P_k}(\upalpha^{w}))=h(P_{\tilde{k}},\mathtt{b}^{\tau})-\phi_{\all}^{\mathtt{b},\tilde{m}}(\pi_{P_{\tilde{k}}}(\upbeta^{\tau})).\end{equation}

    Assuming $k>\dtt/2$, it follows that $\tilde{k}< \dtt/2$. 
    As $\upbeta_{\tilde{m}}\geq 0\geq \upbeta_{\tilde{m}+1}$,
    Proposition~\ref{proposition:3.5} can be evoked to deduce an upper bound $h(G,\mathtt{b})-\sum_{i=1}^{\dtt}\upbeta_i^{+}$ for the RHS in Equation~\eqref{eq:3.2}.
    As we also have
    \[h(G,\att)-\sum_{i=1}^{\dtt}\upalpha_{i}^{+}=h(G,\mathtt{b})-\sum_{i=1}^{\dtt}\upbeta_{i}^{+},\]
    this concludes the proof.
\end{proof}

We can now reduce Theorem~\ref{theorem:1.3} to the following two formal inequalities.
\begin{proposition}\label{proposition:3.7}
    Let $2\leq k\leq \dtt/2$, and $m$ be as before. Let $c$ be as in Lemma~\ref{lemma:new.6}. 
    Assume the following holds:
    \begin{enumerate}
        \item\label{ineq:4.7.1}  If $1\leq \lfloor c\rfloor\leq k-1$:
        \[\sum_{i=1}^{k-\lfloor c\rfloor }c\upalpha_i+\sum_{i=1}^{\lfloor c\rfloor }\Big((\lfloor c\rfloor+1-i)\upalpha_{k-\lfloor c\rfloor +2i-1}+(c-i)\upalpha_{k-\lfloor c\rfloor+2i}\Big)-\sum_{i=1}^{\dtt}\upalpha_i^{+}\geq 0.\]
        \item\label{ineq:4.7.2} 
        If $k \leq \lfloor c\rfloor\leq  \dtt-k-1$:
        \[\sum_{i=1}^{\lfloor c\rfloor -k+1}k\upalpha_i+\sum_{i=1}^{k-1}(k-i)\upalpha_{\lfloor c\rfloor -k+1+2i}+\sum_{i=1}^{k}(k+\{c\}-i)\upalpha_{\lfloor c\rfloor -k+2i}-\sum_{i=1}^{\dtt}\upalpha_{i}^{+}\geq 0\]
    \end{enumerate}
    Then Proposition~\ref{proposition:3.5} holds, and with it Theorem~\ref{theorem:1.3}.
\end{proposition}
\begin{proof}

    By Proposition~\ref{proposition:3.6} and Corollary~\ref{corollary:new.2}, it is sufficient to prove that
    \[\max_{w\in W}(h-\phi_{\all})([w]_{P_k})\leq h(G,\att)-\sum_{i=1}^{\dtt}\upalpha_{i}^{+}\]
    for all $2\leq k\leq \dtt/2$. Fix such $k$.
    Note that $1\leq \lfloor c\rfloor\leq \dtt-k-1$ by Lemma~\ref{lemma:new.6}, and 
    \begin{equation}\label{eq:4.3}\phi_{\all}|_{\Lie(A_{P_k})}=\frac{k(\dtt-k)}{\dtt}c\psi_{k}|_{\Lie(A_{P_k})}\end{equation}
    by Lemma~\ref{lemma:new.1}.
    Consider the upper component of $\partial\mathcal{C}$, which was studied in~\S\ref{subsec:2.2}.
    Note that the line of slope $\frac{k(\dtt-k)}{\dtt}c$ which passes through the point  $\iota(\tau^{\lfloor c\rfloor +1})$ is tangent to $\mathcal{C}$ and bounds it from above, where $\tau^{\lfloor c\rfloor +1}$ is as in Equations~\eqref{eq:3.5N} and~\eqref{eq:3.6N}.
    Therefore, the maximal value of $h-\phi_{\all}$ on $P_k$ is obtained for the element $\tau^{\lfloor c\rfloor +1}$, namely
    \[\max_{w\in W}(h-\phi_{\all})([w]_{P_k})=(h-\phi_{\all})([\tau^{\lfloor c\rfloor+1}]_{P_{k}}).\]

    We would like to compute the difference between the bounds on the entropy in the cusp for the maximal parabolic subgroups $P_{1}$ and $P_{k}$, namely
    \[\Delta=h(G,\att)-\sum_{i=1}^{\dtt}\upalpha_{i}^{+}-(h-\phi_{\all})([\tau^{\lfloor c\rfloor +1}]_{P_{k}})\]
    and show that it is non-negative.
    For simplicity of notation, let $\beta=\upalpha^{\tau^{\lfloor c\rfloor +1}}$. 
    Note that, using Equation~\eqref{eq:4.3}, we have 
    \begin{align*}\Delta &=h(U_{P_k}^{T},\exp(\beta))+\frac{k(\dtt-k)}{\dtt}c\cdot\psi_{k}(\pi_{P_k}(\beta))-\sum_{i=1}^{\dtt}\upalpha_{i}^{+}\\&
    =\sum_{i=k+1}^{\dtt}\sum_{j=1}^{k}(\beta_i-\beta_j)^{+}+c\sum_{j=1}^{k}\beta_j-\sum_{i=1}^{\dtt}\upalpha_i^{+}.\end{align*}

    We separate to two cases depending on the value of $c$.
    First, assume $\lfloor c\rfloor +1\leq k$.
   Then, using Equation~\eqref{eq:3.5N}, and a variation of Lemma~\ref{lemma:3.3N} for $U_{P_k}^{T}$ instead of $U_{P_k}$, we have
   \[\sum_{i=k+1}^{\dtt}\sum_{j=1}^{k}(\beta_i-\beta_j)^{+}
    =\sum_{i=1}^{\lfloor c\rfloor}(\lfloor c\rfloor+1-i)\upalpha_{k-\lfloor c\rfloor +2i-1}-\sum_{i=1}^{\lfloor c\rfloor }i\upalpha_{k-\lfloor c\rfloor +2i},\]
    hence
    \[
    \Delta
    =\sum_{i=1}^{k-\lfloor c\rfloor }c\upalpha_i+\sum_{i=1}^{\lfloor c\rfloor }\Big((\lfloor c\rfloor+1-i)\upalpha_{k-\lfloor c\rfloor +2i-1}+(c-i)\upalpha_{k-\lfloor c\rfloor+2i}\Big)-\sum_{i=1}^{\dtt}\upalpha_i^{+}.\]
Then $\Delta\geq 0$ by assumption, as required.
    
    \medskip
    
    Next, assume $\lfloor c\rfloor +1>k$. Recall again that $\lfloor c\rfloor+1\leq \dtt-k$. 
    Then, using Equation~\eqref{eq:3.6N}, we have
    \begin{align*}\sum_{i=k+1}^{\dtt}\sum_{j=1}^{k}(\beta_i-\beta_j)^{+}&=k\sum_{i=1}^{\lfloor c\rfloor -k+1}\upalpha_i+\sum_{i=1}^{k-1}(k-i)\upalpha_{\lfloor c\rfloor -k+1+2i}-\sum_{i=1}^{k}(\lfloor c\rfloor -k+i)\upalpha_{\lfloor c\rfloor -k+2i}.
    \end{align*}
    Therefore,
    \begin{align*}\Delta&
    =\sum_{i=1}^{\lfloor c\rfloor-k+1}k\upalpha_i+\sum_{i=1}^{k-1}(k-i)\upalpha_{\lfloor c\rfloor-k+1+2i}+\sum_{i=1}^{k}(k+\{c\}-i)\upalpha_{\lfloor c\rfloor-k+2i}-\sum_{i=1}^{\dtt}\upalpha_{i}^{+},
    \end{align*}
    which by assumption is non-negative as well. 
\end{proof}

The rest of this section is devoted to proving these inequalities. We prove a more general formal inequality which takes care of both cases simultaneously. 
First, we require the following simple lemma, whose proof is left to the reader.
\begin{lemma}\label{lemma:4.8}
    Let $(x_i)_{i=1}^{n},(y_i)_{i=1}^{n}$ be monotone non-increasing sequences. 
    Then
    \[\sum_{i=1}^{n}x_i y_i\geq \frac{1}{n}\sum_{i=1}^{n}x_i\cdot\sum_{i=1}^{n}y_i.\]
\end{lemma}

Let us proceed to proving the general formal inequality.
\begin{lemma}\label{lemma:new.4}
    Let $(v_i)_{i=1}^{n}$ be non-negative real numbers for some $1\leq n\leq \dtt$. Let $m$ be some integer so that $\upalpha_{m}\geq 0\geq \upalpha_{m+1}$. Assume one of the following hold:
    \begin{enumerate}
        \item\label{item:?.1} $m<n$, $(v_i)_{i=1}^{n}$ is monotone non-increasing, and
        $\frac{1}{m}\underset{i=1}{\overset{m}{\sum}}(v_i-1)\geq\frac{1}{\dtt-m}\underset{i=m+1}{\overset{n}{\sum}}v_i$.    
        \item\label{item:?.2} $m\geq n$,
        $v_i\geq 1$ for all $i\leq n-1$,
        and $\underset{i=1}{\overset{n}{\sum}}v_i\geq m$.
    \end{enumerate}
    Then 
    \[\sum_{i=1}^{n}v_i \upalpha_i-\sum_{i=1}^{\dtt}\upalpha_i^{+}\geq 0.\]
\end{lemma}
\begin{proof}
    Let us first assume that case~\ref{item:?.1} holds.
    Then, using Lemma~\ref{lemma:4.8},
    \begin{align*}\sum_{i=1}^{n}v_i \upalpha_i-\sum_{i=1}^{\dtt}\upalpha_i^{+}&=\sum_{i=1}^{m}(v_i-1) \upalpha_i+\sum_{i=m+1}^{n}v_i\upalpha_i\\&
    \geq \Big(\frac{1}{m}\sum_{i=1}^{m}( v_i-1)\Big)\sum_{i=1}^{m}\upalpha_i+(\frac{1}{\dtt-m}\sum_{i=m+1}^{n}v_i)\sum_{i=m+1}^{\dtt}\upalpha_i\\&
    =\Big(\frac{1}{m}\sum_{i=1}^{m}( v_i-1)-\frac{1}{\dtt-m}\sum_{i=m+1}^{n}v_i\Big)\sum_{i=1}^{m}\upalpha_i\geq 0,
    \end{align*}
    where we used the fact that $\sum_{i=1}^{\dtt}\upalpha_i=0$.

    Let us now assume that case~\ref{item:?.2} holds.
    Then
    \[\sum_{i=1}^{n}v_i\upalpha_i -\sum_{i=1}^{\dtt}\upalpha_i^{+}=\sum_{i=1}^{n}(v_i-1)\upalpha_{i}-\sum_{i=n+1}^{m}\upalpha_i\geq(\sum_{i=1}^{n}v_i-m)\upalpha_{n}\geq 0.\]
\end{proof}

We consider a specific case, related to both inequalities of Proposition~\ref{proposition:3.7}. 
\begin{proposition}\label{proposition:new.5}
    Let $\kappa\in[0,1)$ and integers $s\geq 1$ and $t\geq 2$. 
    Let $(v_i)_{i=1}^{s+t}$ be defined by
    \[v_i=\begin{cases}
        \kappa & i=s+t \\
        1 &  i=s+t-1 \\
        v_{i+2}+1 & s\leq i\leq s+ t-2 \\
        v_{s} & 1\leq i <s
    \end{cases}.\]
    Let $m$ be as before. 
    Let $S=\sum_{i=1}^{s+t}v_i$. 
    Assume that one of the following holds:
    \begin{enumerate}
        \item $m<s+t\leq \dtt-1$, and $\frac{1}{\dtt-1}(S-v_1)\leq v_1-1$. 
        \item $m\geq s+t$, and  $S\geq m$.
    \end{enumerate}
    Then \[\sum_{i=1}^{s+t}v_i\upalpha_i-\sum_{i=1}^{\dtt}\upalpha_{i}^{+}\geq 0.\]
\end{proposition}

\begin{proof}
    Note that if $m\geq s+t$ and $S\geq m$ as in second condition, then Item~\ref{item:?.2} of Lemma~\ref{lemma:new.4} applies and the result follows immediately.
    
    So we assume that the first condition, for the $m<s+t$ case, holds. 
    We separate to two cases. 
    First of all, assume $m>s$. Then $m\geq 2$, and we have by monotonicity
    \[\frac{1}{m}\sum_{i=1}^{m}(v_i-1)\geq \frac{v_{m-1}+v_{m}}{2}-1.\]
    As $m< s+t\leq \dtt-1$, we also similarly have
    \[\frac{1}{\dtt-m}\sum_{i=m+1}^{s+t}v_i\leq \begin{cases}\frac{v_{m+1}+v_{m+2}}{2} & \text{if }\ \ m+1<s+t\\
    \frac{1}{2}v_{m+1} & \text{if }\ \ m+1=s+t
    \end{cases}=\frac{v_{m-1}+v_{m}}{2}-1.\] 
    Together we have
    \[\frac{1}{m}\sum_{i=1}^{m}(v_i-1)\geq \frac{1}{\dtt-m}\sum_{i=m+1}^{s+t}v_i\]
    and so the result follows for this case from item~\ref{item:?.1} of Lemma~\ref{lemma:new.4}.
    
    Otherwise, if $m\leq s$, then by monotonicity and positivity
    \begin{align*}\frac{1}{\dtt-m}\sum_{i=m+1}^{s+t}v_i&
    \leq \frac{1}{\dtt-m}\frac{s+t-m}{s+t-1}\sum_{i=2}^{s+t}v_i
    \leq \frac{1}{\dtt-1}\sum_{i=2}^{s+t}v_i
    =\frac{1}{\dtt-1}(S-v_1)\\&
    \leq v_1-1=\frac{1}{m}\sum_{i=1}^{m}(v_i-1).
    \end{align*}
    So the result follows from  Lemma~\ref{lemma:new.4} as before.
\end{proof}

\begin{lemma}\label{lemma:3.9}
    The sum of coefficients $S$ as in Proposition~\ref{proposition:new.5} satisfies
    \[S=\begin{cases}
    \frac{t+1}{2}(\kappa+s+\frac{t-1}{2}) & \text{$t$ is odd}\\
    (\frac{t}{2}+s)(\frac{t}{2}+\kappa) & \text{$t$ is even}
    \end{cases}\]
\end{lemma}

\begin{proof}
    In case $t$ is odd, we have
    \begin{align*}
        \sum_{i=1}^{n}v_i&=\sum_{i=0}^{(t-1)/2}(\kappa+i)+\sum_{i=1}^{(t-1)/2}i+\sum_{i=1}^{s}\frac{t+1}{2}=\frac{t+1}{2}\kappa+2\sum_{i=1}^{(t-1)/2}i+s\frac{t+1}{2}\\&
        =\frac{t+1}{2}(\kappa+s)+\frac{t+1}{2}\frac{t-1}{2}=\frac{t+1}{2}(\kappa+s+\frac{t-1}{2}).
    \end{align*}
    In case $t$ is even, we have
    \begin{align*}
        \sum_{i=1}^{n}v_i&=\sum_{i=0}^{t/2-1}(\kappa+i)+\sum_{i=1}^{t/2}i+\sum_{i=1}^{s}(\kappa+\frac{t}{2})=\frac{t}{2}\kappa+2\sum_{i=1}^{t/2-1}i+s\kappa+(s+1)\frac{t}{2}\\&
        =\frac{t}{2}\kappa+\frac{t}{2}(\frac{t}{2}-1)+s\kappa+(s+1)\frac{t}{2}=(\frac{t}{2}+s)(\frac{t}{2}+\kappa).
    \end{align*}
\end{proof}

We can now conclude.
\begin{proof}[Proof of Theorem~\ref{theorem:1.3}]
    We prove that the two inequalities in Proposition~\ref{proposition:3.7} hold.
    Both inequalities are of the form \[\sum_{i=1}^{s+t}v_i\upalpha_i-\sum_{i=1}^{\dtt}\upalpha_{i}^{+}\geq 0\] for $v_i$ as in Proposition~\ref{proposition:new.5}, where $\kappa=\{c\}$, and  $t=2\lfloor c\rfloor,\ s=k-\lfloor c\rfloor$ for the first inequality ~\ref{proposition:3.7}~\eqref{ineq:4.7.1}, while 
    $t=2k-1,\ s=\lfloor c\rfloor -k+1$ for the second~\ref{proposition:3.7}~\eqref{ineq:4.7.2}.
    
    Then, it follows from Lemma~\ref{lemma:3.9} that in either case
    $S=kc$. Therefore, by Lemma~\ref{lemma:new.6}, we have $S\geq m$. Then, to use Proposition~\ref{proposition:new.5}, it is sufficient to show (in both cases) that \[\frac{1}{\dtt-1}(S-v_1)\leq v_1-1.\]
    
    For the first inequality~\ref{proposition:3.7}~\eqref{ineq:4.7.1}, we have $v_1=c$, and so
    \[\frac{1}{\dtt-1}(S-v_1)=\frac{1}{\dtt-1}(k-1)c=c-\frac{\dtt-k}{\dtt-1}c\leq c-1=v_1-1\]
    as required, where we used Lemma~\ref{lemma:new.6}.
    
    Next, similarly, for the second inequality~\ref{proposition:3.7}~\eqref{ineq:4.7.2} we have $v_1=k$, and by using Lemma~\ref{lemma:new.6} for $kc\leq\dtt-1$, we obtain
    \[\frac{1}{\dtt-1}(S-v_1)=\frac{1}{\dtt-1}(kc-k)\leq 1-\frac{k}{\dtt-1}<1\leq k-1=v_1-1,\]
    as required. This concludes the proof of the theorem.
\end{proof}

\bibliographystyle{plain}

\end{document}